\documentclass{article}

\usepackage{geometry}
\newgeometry{vmargin={1in}, hmargin={1.25in,1.25in}}
\usepackage{microtype} 
\usepackage{graphicx}
\usepackage[makeroom]{cancel}
\usepackage{booktabs} 
\usepackage{makecell}
\usepackage[colorlinks=true,citecolor=blue]{hyperref}
\usepackage{multirow}
\usepackage[table, svgnames, dvipsnames]{xcolor}
\usepackage{makecell, cellspace, caption}
\usepackage{booktabs,caption}
\usepackage[flushleft]{threeparttable}
\usepackage{amsmath, nccmath}
\usepackage{amssymb}
\usepackage{mathtools}
\usepackage{enumitem}
\usepackage{amsthm}
\usepackage{lipsum}
\usepackage[capitalize,noabbrev]{cleveref} 

\usepackage{amsfonts}      
\usepackage{cite}
\usepackage{xcolor}
\usepackage{algorithm}
\usepackage{algpseudocode}
\usepackage{dsfont}
\usepackage{multicol}
\usepackage{cuted} 
\usepackage{booktabs} 
\usepackage{pifont} 
\usepackage{footnote} 
\makesavenoteenv{table} 
\usepackage{threeparttable}  
\usepackage{subcaption}
\usepackage{amssymb} 
\usepackage{caption}
\allowdisplaybreaks

\newtheorem{theorem}{Theorem}
\newtheorem{corollary}{Corollary}
\newtheorem{lemma}{Lemma}
\newtheorem{assumption}{Assumption}
\newtheorem{proposition}{Proposition}
\newtheorem{remark}{Remark}
\newtheorem{example}{Example} 

\def\R{\mathbb{R}}
\def\E{\mathbb{E}}

\def\P{\mathbb{P}}

\def\I{\mathbf{I}}

\newcommand{\bb}{\boldsymbol{b}}

\newcommand{\e}{\boldsymbol{e}}

\newcommand{\m}{\boldsymbol{m}}

\newcommand{\br}{\boldsymbol{r}}

\newcommand{\bu}{\boldsymbol{u}}
\newcommand{\bv}{\boldsymbol{v}}
\newcommand{\w}{\boldsymbol{w}}
\newcommand{\x}{\boldsymbol{x}}
\newcommand{\y}{\boldsymbol{y}}
\newcommand{\z}{\boldsymbol{z}}

\def\bA{\boldsymbol{A}}
\def\bB{\boldsymbol{B}}
\def\bC{\boldsymbol{C}}

\def\bH{\boldsymbol{H}}
\def\bI{\boldsymbol{I}}

\def\bQ{\boldsymbol{Q}}

\def\bS{\boldsymbol{S}}

\def\bW{\boldsymbol{W}}
\def\bX{\boldsymbol{X}}

\def\cE{\mathcal{E}}
\def\cF{\mathcal{F}}

\def\cO{O}

\def\cV{\mathcal{V}}

\newcommand{\msf}{\mathsf}

\DeclareMathOperator{\clip}{\textsf{clip}}

\def\diag{\msf{Diag}}

\newcommand{\bxi}{\boldsymbol{\xi}}
\newcommand{\bet}{\boldsymbol{\eta}}
\newcommand{\bPhi}{\boldsymbol{\Phi}}
\newcommand{\bPsi}{\boldsymbol{\Psi}}

\newcommand{\kaph}{\kappa_{\phi}}

\newcommand{\cb}{c_\beta}
\newcommand{\cp}{c_\varphi}
 
\newcommand{\cet}{c_\eta}

\newcommand{\gd}{\nabla}

\newcommand{\one}{\mathbf{1}}
\newcommand{\zero}{\mathbf{0}}
\newcommand{\ox}{\bar{\x}}

\newcommand{\sen}{\operatorname{SClip-EF-Network}}

\newcommand{\argmin}{\mathop{\mathrm{arg\,min}}}

\newcommand{\norm}[1]{\|#1\|}

\newcommand{\cmark}{\text{\ding{51}}}
\newcommand{\xmark}{\text{\ding{55}}}

\newcommand{\sce}{\text{SClip-EF}}
\newcommand{\scen}{\text{SClip-EF-Network}} 
\newcommand{\pcss}{\text{Prox-clipped-SGD-shift}}

\title{
        Smoothed Gradient Clipping and Error Feedback for Decentralized Optimization under Symmetric Heavy-Tailed Noise
}

\author{Shuhua Yu\thanks{
     Department of Electrical and Computer Engineering,
     Carnegie Mellon University, Pittsburgh, PA 15213, USA. Emails:
     \texttt{\{shuhuay, soummyak\}@andrew.cmu.edu}}
     \and
     Du\u san Jakoveti\'c\thanks{
     Faculty of Sciences, Department of Mathematics and Informatics, University of Novi Sad, Novi Sad, 21000, Serbia. Email: \texttt{dusan.jakovetic@dmi.uns.ac.rs}}
     \and
     Soummya Kar\footnotemark[1]} 
\date{November 11, 2024}

\begin{document}
\maketitle
\vskip 0.3in

\begin{abstract}
Motivated by understanding and analysis of large-scale machine learning under heavy-tailed gradient noise, we study decentralized optimization with gradient clipping, i.e., in which certain clipping operators are applied to the gradients or gradient estimates computed from local nodes prior to further processing. While vanilla gradient clipping has proven effective in mitigating the impact of heavy-tailed gradient noise in non-distributed setups, it incurs bias that causes convergence issues in heterogeneous distributed settings. To address the inherent bias introduced by gradient clipping, we develop a smoothed clipping operator, and propose a decentralized gradient method equipped with an error feedback mechanism, i.e., the clipping operator is applied on the difference between some local gradient estimator and local stochastic gradient. We consider strongly convex and smooth local functions under symmetric heavy-tailed gradient noise that may not have finite moments of order greater than one. We show that the proposed decentralized gradient clipping method achieves a mean-square error (MSE) convergence rate of $\cO(1/t^\delta)$, $\delta \in (0, 2/5)$, where the exponent $\delta$ is independent of the existence of higher order gradient noise moments $\alpha > 1$ and lower bounded by some constant dependent on condition number. To the best of our knowledge, this is the first MSE convergence result for decentralized gradient clipping under heavy-tailed noise without assuming bounded gradient. Numerical experiments validate our theoretical findings. 
\end{abstract}
\section{Introduction} 
In this paper, we address heterogeneous decentralized optimization under heavy-tailed \footnote{By heavy-tailed distribution, we mean that the noise distribution has a heavier tail than any exponential distribution, i.e., a random variable $X$ is called heavy-tailed if for any constant $a > 0$, $\limsup_{x \rightarrow \infty} \P(X > x)e^{ax} = \infty$  \cite{nair2022fundamentals}.} gradient noise. We consider a network of $n$ nodes where each node holds a local function $f_i$ and accesses its stochastic gradient with heavy-tailed additive noise via a local oracle. Collectively, these nodes aim to minimize the aggregated function $f :=(1/n)\sum_{i=1}^n f_i$. We assume a general decentralized network where nodes communicate only with their direct neighbors, according to a predefined network topology. Specifically, when the network topology is a complete graph, this decentralized setup simplifies to a server-client setup, where one node acts as a server and all other nodes communicate directly with it. 




The above distributed average-cost minimization formulation has become a key approach in the emerging distributed machine learning paradigm \cite{xin2020general, kairouz2021advances}, as a prominent method of handling possibly distributed large-scale modern machine learning models and datasets. Significant research efforts to address this formulation have focused on decentralized gradient methods \cite{tsitsiklis1986distributed}, including decentralized (sub)gradient descent \cite{sundhar2010distributed, koloskova2020unified}, accelerated gradient methods \cite{jakovetic2014fast}, variance reduction schemes \cite{shi2015extra}, and gradient tracking-based methods \cite{di2016next, pu2021distributed}, among others. Most prior art has addressed gradient noise with bounded variance, covering some heavy-tailed distributions such as log-normal and Weibull; however, there has been limited investigation into heavy-tailed gradient noise without the bounded variance assumption. It is well known that classical centralized and distributed stochastic optimization procedures are likely to suffer from instability and divergence in the presence of heavy-tailed stochastic gradients \cite{zhang2020adaptive, jakovetic2023nonlinear, al2016robust}. Further, in decentralized setups, heavy-tailed noise can cause stochastic gradients to vary significantly in magnitude, posing implementation challenges such as directly communicating stochastic gradients with sufficient precision \cite{jakovetic2023distributed, vukovic2024nonlinear}. To mitigate these issues, distributed gradient clipping based approaches, that use appropriately clipped stochastic gradients at local nodes to combat heavy-tailed noise, have been proposed in recent works \cite{yang2022taming, gorbunovhigh, sun2024distributed}. In the server-client setup: the work \cite{yang2022taming} assumes that the stochastic gradients have uniformly bounded $\alpha$ moments for some $\alpha \in (1, 2]$ to show MSE convergence, which essentially imposes the restrictive conditions of uniformly bounded gradients and noise distributions possessing $\alpha$ moments for some $\alpha > 1$; another concurrent work \cite{gorbunovhigh} relaxes the bounded gradient condition but still assumes that gradient noise has bounded $\alpha$ moment for $\alpha \in (1, 2]$, in establishing  high-probability convergence. In the more general decentralized setup, \cite{sun2024distributed} also addresses gradient noise with $\alpha$-moment for $\alpha \in (1, 2]$ in proving almost sure (a.s.) convergence, but additionally assumes a compact domain which also implies bounded gradients. Our work considers the more general decentralized setup, which, under a fully connected network, subsumes the server-client setup as a special case with the same algorithmic steps. Notably, by assuming noise symmetry, we relax the bounded gradient assumption (that enables us to deal with scenarios such as strongly convex costs) and, instead of requiring a uniform $\alpha \in (1, 2]$ moment bound on the nodes' stochastic gradients, we only assume that the gradient noise have bounded first absolute moments, i.e., $\alpha = 1$, to establish MSE convergence.  

The study of heavy-tailed gradient noise is motivated by empirical evidence arising in the training of deep learning models \cite{csimcsekli2019heavy, simsekli2019tail, zhang2020adaptive}. For example, the distribution of the gradient noise arising in the training of attention models resembles a Levy $\alpha$-stable distribution \cite{nair2022fundamentals} that has unbounded variance \cite{gurbuzbalaban2021heavy}. In addition, there is evidence that heavy-tailed noise can be induced by distributed machine learning systems with heterogeneous datasets that are not independently and identically distributed (i.i.d.) \cite{yang2022taming}. Recently, there have been significant advances in developing algorithms and analyses for non-distributed stochastic optimization under certain variants of heavy-tailed noise. The work \cite{zhang2020adaptive} analyzes clipped stochastic gradient descent (SGD) and establishes convergence rates in the mean-squared sense under gradient noise with bounded $\alpha$-moment, $\alpha \in (1, 2]$, and \cite{csimcsekli2019heavy} shows mean-squared convergence of vanilla SGD under heavy-tailed $\alpha$-stable distribution. More related to this work, in \cite{jakovetic2023nonlinear}, the authors also assume symmetric heavy-tailed gradient noise that has bounded first absolute moment (and possibly no moment of order strictly greater than one), and introduce a general analytical framework for \textit{nonlinear} SGD which subsumes popular choices such normalized gradient descent, clipped gradient descent, quantized gradient, etc., proving MSE, a.s.\ convergence and asymptotic normality. This work, focusing on decentralized setups and MSE convergence analysis, considers similar heavy-tailed noise in \cite{jakovetic2023nonlinear}, i.e., \textit{symmetric} noise with $\alpha$-moment bound only for \textit{$\alpha = 1$}. Note that, utilizing tools such as a generalized centralized limit theorem, several studies have shown that \textit{symmetric} heavy-tailed gradient noise, such as symmetric $\alpha$-stable distributions, appropriately characterizes noise models in various neural network training processes \cite{simsekli2019tail, peluchetti2020stable, gurbuzbalaban2021heavy, barsbey2021heavy}, which are also supported by empirical experiments in an array of neural network architectures and datasets \cite{barsbey2021heavy, battash2024revisiting}. In terms of error analysis, there is also a growing interest in deriving high-probability rates that \textit{logarithmically} \footnote{With in-expectation convergence guarantees, one can use Markov's inequality to show high-probability rates with confidence level $\beta$ that have inverse-power dependence $\cO(1/\beta)$.} depend on the confidence level under heavy-tailed gradients for clipped SGD type methods in non-distributed setups \cite{nazin2019algorithms, puchkin2023breaking, davis2021low, gorbunov2020stochastic, gorbunov2021near, zhang2022parameter, cutkosky2021high, sadiev2023high}, and in server-client setups \cite{gorbunovhigh}. We defer more discussions on high-probability convergence analyses to Section \ref{sec:otherworks}.

While gradient clipping and related techniques have been extensively studied in various non-distributed scenarios besides heavy-tailed noise, including differential privacy \cite{Dwork2006CalibratingNT}, $(L_0, L_1)$-smooth functions \cite{Zhang2020Why, koloskova2023revisiting}, etc., extending these analyses to \textit{distributed} setups is highly non-trivial and thus the existing analyses of \textit{distributed} gradient clipping are under relatively strong assumptions. In \cite{koloskova2023revisiting}, the authors show that under gradient noise with bounded variance, gradient clipping with constant clipping threshold incurs unavoidable bias \footnote{By bias, we mean the difference between the expected clipped gradients and the true gradients.} resulting in convergence to non-stationary points, which extends to the differential private SGD where clipped i.i.d. stochastic gradients are summed up before used for model updates. In this work, we use \textit{decaying} clipping thresholds to circumvent this issue. In the distributed differential privacy setup with \textit{heterogeneous} functions: the authors of \cite{zhang2022understanding} prove  convergence of the distributed gradient clipping in the server-client setup, by assuming \textit{bounded} stochastic gradient and uniformly bounded gradient dissimilarity, i.e., $\|\nabla f_i(\x) - \nabla f(\x) \| \le c, c > 0,  \forall i \in [n], \forall \x$; and the authors of \cite{li2023convergence} establish convergence in the decentralized setup by assuming $\| \nabla f_i(\x) - \nabla f(\x) \| \le (1/12) \|\nabla f(\x)\|, \forall i \in [n], \forall \x$. In adversarial-robust decentralized optimization where gradient clipping is adopted to combat adversarial noise, the authors of \cite{yu2023secure} establish convergence of decentralized gradient clipping in a special case where all local convex functions share the same minimizer and $\sum_{i=1}^n f_i$ is strongly convex. To handle $(L_0, L_1)$ smooth functions using distributed gradient clipping in server-client setups, the work \cite{crawshaw2024federated} assumes that $\| \nabla f_i(\x) \| \le a + b \| \nabla f(\x)\|, a \ge 0, b \ge 1,  \forall i \in [n], \forall \x$, and almost surely \textit{bounded} gradient noise. In contrast, this work considers a general decentralized setup, and assumes \textit{bounded gradient heterogeneity only at the global optimum}, i.e., let $ \x^* = \argmin_{\x \in \R^d} f(\x)$ we assume $\| \nabla f_i(\x^*) \|_\infty \le c_* $ \footnote{Note that such $c_*$ naturally exists given the existence of an optimum.} for each $i \in [n]$ and $c_* > 0$, which is strictly weaker than all aforementioned heterogeneity assumptions, and does not require gradient or gradient noise to be bounded. The relaxation of the above assumptions is made possible by our proposed algorithmic approach with error-feedback mechanism, i.e., applying a smoothed clipping operator on the estimating errors between local stochastic gradients and a momentum-type local gradient estimator. Similar algorithmic constructions with error-feedback are also used in recent distributed clipping works including \cite{khirirat2023clip21} in deterministic setting, and \cite{gorbunovhigh} under heavy-tailed gradient noise, both in sever-client setups. 
 

\begin{table*}[t] 
  \centering
  \begin{threeparttable} 
    \begin{tabular}{ccccccc}
      \toprule
  Methods & SC? & D? & UG? & SM? & $\alpha$ & MSE rates \\
  \midrule
  GClip \cite{zhang2020adaptive} &  \xmark & \xmark & \xmark & \xmark & $\alpha > 1$ & $\cO(1/t^{2- 2/\alpha })$ \\
  Dist-GClip \cite{yang2022taming} & \cmark & \xmark & \xmark & \xmark &  $\alpha > 1$ & $\cO(1/t^{2- 2/\alpha })$ \\ 
  \cite{sun2024distributed} & \cmark & \cmark & \xmark & \xmark & $\alpha > 1$ & Inapplicable\tnote{1} \\
    Ours & \cmark & \cmark & \cmark & \cmark &  $\alpha \ge 1$ & $\cO(1/t^{\min(c_s, \ 2/5)})$ \tnote{2} \\
  \bottomrule
    \end{tabular}
    \begin{tablenotes}
    {\footnotesize
        \item[1] The work \cite{sun2024distributed} proves  a.s.\ convergence without rate. 
      \item[2] $c_s$ is a constant independent of $\alpha$.
      }
    \end{tablenotes}
  \end{threeparttable} 
  \caption{A comparison on the assumptions and MSE rates for gradient methods with clipping for strongly convex functions. Column ``SC?" shows whether the method is for server-client distributed setup, ``D?" shows whether the method is for decentralized network setups,  ``SM'' means whether noise is symmetric, ``UG" shows whether the analysis works under unbounded gradient, and column $\alpha$ shows the allowed values of $\alpha$ for the noise assumption $\E\| \bxi \|^\alpha < \infty$.}
  \label{tab:ass_rates}
\end{table*}

\subsection{Contributions} 
Our principal contributions are delineated as follows: 


\begin{enumerate}
    \item We establish the first MSE convergence rate in heterogeneous \textit{decentralized} optimization under heavy-tailed gradient noise. Specifically, we address the case where local functions are strongly convex and smooth, subject to heavy-tailed gradient noise that is component-wise \textit{symmetric about zero} and has \textit{first absolute moment}. Whereas, in server-client setups, the most pertinent work \cite{yang2022taming}  assumes that for some $\alpha \in (1, 2], \ c > 0$, $\E \| \nabla f(\x) + \bxi \|^{\alpha} \le c$ uniformly in $\x$, where $\nabla f(\x)$ denotes gradient and $\bxi$ denotes noise that is \textit{zero-mean}, and shows $\cO\big(1/t^{2 - 2/\alpha}\big)$ MSE convergence rate. Apart from the restricted \textit{bounded gradient condition,} the MSE rate $\cO\big(1/ t^{2 - 2/\alpha}\big)$ derived in \cite{yang2022taming} approaches $\cO(1)$ when $\alpha$ approaches 1, rendering it inapplicable in our considered scenario in which no moment greater than $\alpha=1$ is required to exist. In general decentralized setups, the most relevant work \cite{sun2024distributed}, assuming $\E \| \bxi \|^\alpha \le c$ for $\alpha \in (1, 2]$, shows a.s.\ convergence without rate under a \textit{bounded domain} condition, which our work does not require. Our work addresses these limitations with distinct algorithmic constructions and analyses in general \textit{decentralized} setups. We establishes an MSE rate $\cO(1/t^{\delta})$ for any positive $\delta < \min(c_s, 2/5)$ where $c_s$ is independent of the existence of $\alpha$ moments with $\alpha >1$, and is lower bounded by some constant dependent on the condition number, as specified in Theorem \ref{thm:bdd_alpha}. We summarize these  comparisons of assumptions and rates in Table \ref{tab:ass_rates}.
    
    

    \item To tackle the inherent bias arising from gradient clipping in the \textit{stochastic} and \textit{heterogeneous} decentralized case, we adopt the algorithmic approach that integrates smoothed gradient clipping and error feedback, i.e., clipping the gradient estimation errors rather than stochastic gradients. While the authors of \cite{khirirat2023clip21} and \cite{gorbunovhigh} also develop variants of error feedback schemes for distributed gradient clipping in heterogeneous server-client setups, addressing deterministic gradient and stochastic gradient with heavy-tailed noise, respectively, this work considers the broader \textit{decentralized} network setups. Moreover, we carefully craft accompanying \textit{smooth} clipping operators and \textit{weighted} error feedback to mitigate the effect of symmetric heavy-tailed noise. This specific design leads to substantial differences in analysis with respect to \cite{khirirat2023clip21} and \cite{gorbunovhigh}, allowing for the relaxation of the bounded $\alpha$-moment assumption in \cite{gorbunovhigh} to the case where no moment greater than one is required to exist, and constitutes a distinct technical  contribution.

    
\end{enumerate} 

\subsection{Related work}
\label{sec:otherworks}
\textit{High-probability convergence in non-distributed setups}. For the case where gradient noise may have unbounded variance, in \cite{zhang2022parameter, cutkosky2021high}, the authors establish high-probability rates for convex and nonconvex functions assuming bounded gradients. Recently, the authors of \cite{sadiev2023high} and \cite{nguyen2023improved}, relax the restrictive ``bounded gradient'' assumption and prove high-probability rates with optimal dependence on $\alpha \in (1, 2]$, for strongly convex and non-convex smooth functions, respectively. In \cite{puchkin2023breaking}, leveraging additional structures in gradient noise, the authors propose a nonlinear operator coined smoothed medians of means, which requires multiple stochastic gradient samples at a single point, and manage to derive high-probability rates independent of noise moment $\alpha$. In \cite{armacki2024nonlinear}, by assuming symmetric gradient noise with positive probability mass around zero, the authors establish high-probability convergence for a broad class of nonlinear operators, including clipping and normalization, without imposing any moment bound on gradient noise. 



\subsection{Paper organization}
In Section \ref{sec:model}, we introduce our problem model and assumptions. In Section \ref{sec:alg}, we develop our algorithms and present the main results in Section \ref{sec:results} with convergence  analysis in Section \ref{sec:proof}. Section \ref{sec:experiments} demonstrates the effectiveness of our algorithm by a strongly convex example with synthetic data. Section \ref{sec:con} concludes our paper with some discussions on limitations and directions for future research. 



\subsection{Notations}
We denote by $[n]$ the set $\{1, 2, \ldots, n\}$; by $\R$ and $\R_{+}$, respectively, the set of real numbers and nonnegative real numbers, and by $\R^d$ the $d$-dimensional Euclidean space. We use lowercase normal letters for scalars, lowercase boldface letter for vectors, and uppercase boldface letter for matrices. Further, we denote by $x_i$ the $i$th component of $\x$; by  $|\x|$ the vector $[|x_1|, \ldots, |x_d|]^\top$ for $\x \in \R^d$; by $\diag(\x)$ the diagonal matrix whose diagonal is the argument vector $\x$; by $\one $ and $\zero$ the all-ones and all-zeros vector, respectively; and by $\bI$ the identity matrix. Next, we let  $\| \x \|$ and $\| \x \|_\infty$ denote the $\ell_2$ norm and infinity norm of vector $\x$, respectively; and $\| \bX \|_2$ denote the operator norm of $\bX$. For functions $p(t), q(t)$ in $t$, we have $p(t) = \cO(q(t))$ if $\limsup_{t \rightarrow \infty} p(t)/q(t) < \infty$, $p(t) = \Theta(q(t))$ if there exist positive $c_1, c_2, t_0$ such that $\forall t \ge t_0, 0 \le c_1 q(t) \le p(t) \le c_2 q(t)$, and $p(t) = \Omega(q(t))$ if there exist positive $c, t_0$ such that $\forall t \ge t_0, p(t) \ge cq(t) \ge 0$; for event $A$, the indicator function $\mathds{1}_A = 1$ if event $A$ happens, otherwise $0$; and we use $\E$ for expectation over random events. Finally, any inequality between two vectors, or between a vector and a scalar holds true on a component-wise basis; an inequality between random variables and scalars holds true almost surely; an inequality $\bA \succeq \bB$ holds if $\bA - \bB$ is positive semidefinite.  
\section{Problem Model} 
\label{sec:model}
Consider a network of $n$ nodes where each node $i$ holds a local and private function $f_i$. 

\begin{assumption} 
\label{as:func}
For each node $i \in [n]$, the local objective function $f_i: \R^d \rightarrow (-\infty, \infty)$ is twice-continuously differentiable, $\mu$-strongly convex, and $L$-smooth. That is, $\forall i \in [n], \ \forall \x \in \R^d$, the Hessian matrix of $f_i$ at $\x$ satisfies $\mu \bI \preceq \nabla^2 f_i(\x) \preceq L \bI$ for some constants $L \ge \mu > 0$. 
\end{assumption}



The following assumption measures the heterogeneity among local functions and is naturally satisfied for any finite $n$. 
\begin{assumption}
\label{as:bdd_hete} 
Let $\x^* := \argmin_{\x \in \R^d} f(\x)$. For any node $i \in [n]$, the gradient of local function $f_i$ at the global minimizer $\x^*$, satisfies $\| \nabla f_i(\x^*) \|_\infty \le c_*$. 
\end{assumption}

Let the stochastic gradient for $f_i$ at query $\x_i^t$ be $g_i(\x_i^t)$ such that $g_i(\x_i^t) = \nabla f_i(\x_i^t) + \bxi_i^t$,  where $\bxi_i^t$ denotes the stochastic gradient noise at $\x_i^t$. 
\begin{assumption}
\label{as:ht-noise} 
The sequence of gradient noise vectors  $\{ \{\bxi_i^t\}_{i \in [n]} \}_{t \ge 0}$ are independently distributed over all iterations $t \ge 0$. For any $t \ge 0, i \in [n]$, we assume that every component of the random vector $\bxi_i^t \in \R^d$ has the identical marginal probability density function (pdf), denoted as $p: \R \rightarrow \R_+$, that satisfies the following:  
\begin{enumerate}
    \item $p$ is symmetric about zero, i.e.,  $p(u) = p(-u), \forall u \in \R$.

    \item \label{subas:1} $p$ has bounded first absolute moment, i.e., there exists some constant $\sigma > 0$ such that $\int_{-\infty}^{\infty} |u| p(u) du \le \sigma$.

    
\end{enumerate}
\end{assumption} 

\begin{remark}
The conditions in Assumption \ref{as:ht-noise} ensure the existence of a joint pdf \(\hat{p}: \R^d \rightarrow \R_+\) for any \(\bxi_i^t\). This can be relaxed to allow $\hat{p}$ to have \textit{non-identical} marginal pdfs, although we assume they are equal for notational simplicity. 
\end{remark}
 
Example distributions that satisfy Assumption \ref{as:ht-noise} include light-tailed zero-mean Gaussian distribution, zero-mean Laplace distribution, and also heavy-tailed zero-mean symmetric $\alpha-$stable distribution \cite{bercovici1999stable}. 

\begin{example}
\label{exm:dist}
For a noise distribution satisfying Assumption \ref{as:ht-noise} but not the assumptions used in \cite{yang2022taming, sadiev2023high}, consider the following pdf (see Figure \ref{fig:ht-example} for plot for $\ln \big(\P(u > x)\big)$ and its comparison with a Gaussian with the same value of $\int_{-\infty}^\infty |u| p(u) du$.),
\begin{align} 
\label{eq:ht-example}
    p(u) = \frac{c_p}{(u^2 + 2) \ln^2(u^2 + 2)}, 
\end{align} 
where $c_p := 1/\int_{-\infty}^{\infty} 1/[(u^2 + 2) \ln^2(u^2 + 2)] du$. We have the following proposition with proofs deferred to Appendix \ref{sec:example}. 
\end{example}

\begin{figure}[htbp]
    \centering
    \begin{subfigure}[b]{0.45\textwidth}
        \centering \includegraphics[width=\textwidth]{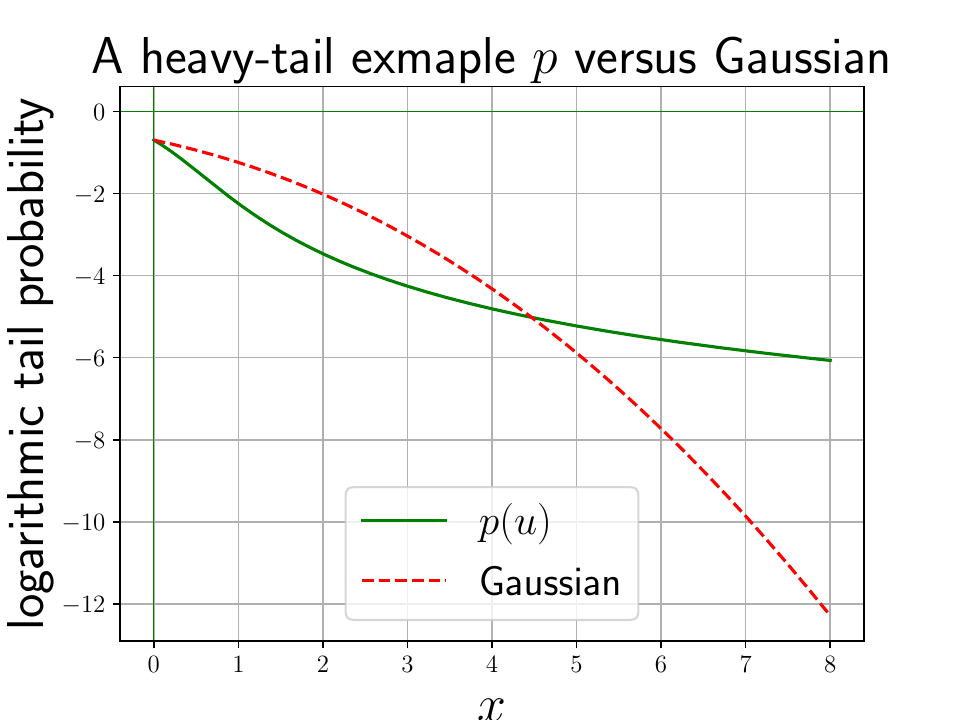}
        \caption{An example of a heavy-tailed distribution.}
        \label{fig:ht-example}
    \end{subfigure}
    \hfill
    \begin{subfigure}[b]{0.45\textwidth}
        \centering
        \includegraphics[width=\textwidth] {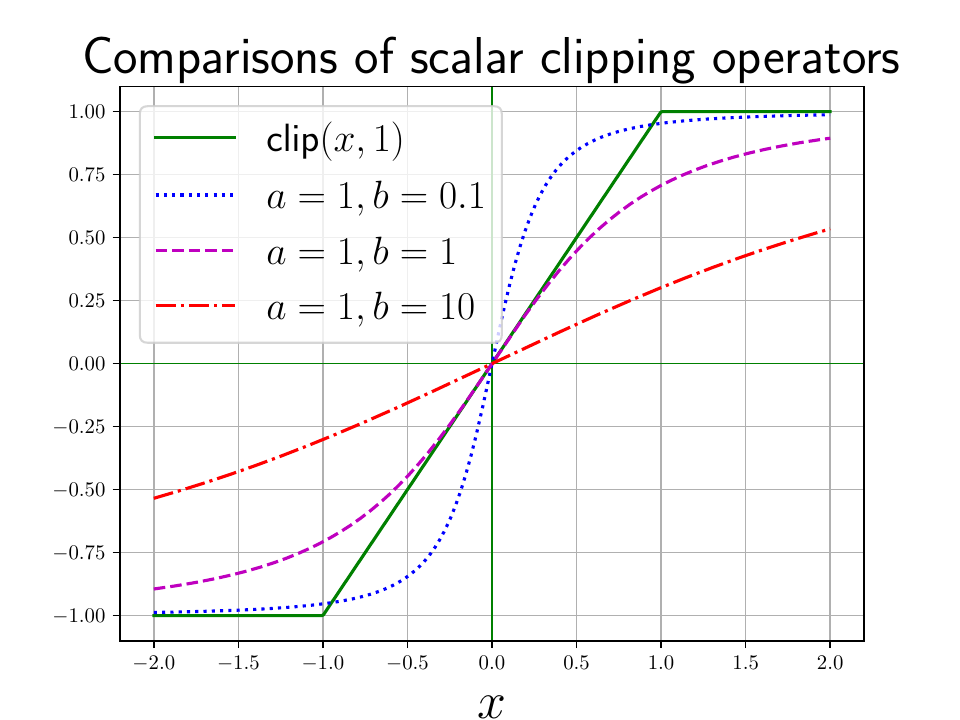}
        \caption{An illustration of $\Psi(x) = \frac{ax}{\sqrt{x^2 + b}}$}  
    \label{fig:sc-operator}
    \end{subfigure}
\end{figure}

\begin{proposition} 
\label{prop:example}
    The distribution pdf in Example \ref{exm:dist} is heavy-tailed, has bounded first absolute moment, but does not have any moment greater than one.  
\end{proposition}
\section{Algorithm development}
\label{sec:alg}
In the non-distributed setup, several works including \cite{zhang2020adaptive, koloskova2023revisiting, sadiev2023high, jakovetic2023nonlinear} have analyzed the clipped SGD algorithm,
\begin{align}
\label{eq:clipsgd}
\x^{t+1} = \x^t - \eta_t \clip_\lambda ( g(\x^t)),
\end{align}
where $g(\x^t)$ is the stochastic gradient at $\x^t$ and $\eta_t$ is the stepsize. The clipping operator is defined as $\clip_\lambda (\y) = \min(\lambda/|\y|, 1) \y, \forall \y \in \R^d$, or component-wise, $[\clip_\lambda (\y)]_i = \min(\lambda, [\y]_i)$ $\forall i = 1, \ldots, d$, and the corresponding update \eqref{eq:clipsgd} is denoted as Global Clipping (GClip) and Component-wise Clipping (CClip) in \cite{zhang2020adaptive}, respectively. A natural extension in the server-client distributed setup is for each node to clip the local stochastic gradient before sending it to the server, resulting in the following update:
\begin{align}
\label{eq:clippedsgd_client}
    \x^{t+1} = \x^t - \frac{\eta_t}{n} \sum_{i=1}^n \clip(g_i(\x^t), \lambda), 
\end{align} 
which we designate as Dist-GClip or Dist-CClip, depending on the choice of clipping operator. Similarly, in the decentralized setup with mixing matrix $\{w_{ij}\}$, one has 
\begin{align}
\label{eq:network_clippings}
    \x_i^{t+1} & = \sum_{j=1}^n w_{ij} \big(\x_i^t - \eta_t \clip( g_j(\x_j^t), \lambda)\big),
\end{align}
and we refer \eqref{eq:network_clippings} as Network-GClip or Network-CClip based on the selected clipping method. However, for any constant $\lambda > 0$, the above update \eqref{eq:clippedsgd_client}  can lead to convergence to non-stationary points due to stochastic bias \cite{koloskova2023revisiting} and node heterogeneity \cite{zhang2022understanding}. 

In the server-client distributed setup, we address this issue by proposing an algorithm that utilizes a Smooth Clipping operator with decaying clipping threshold and a weighted Error Feedback (SClip-EF). We use clipping to combat heavy-tailed stochastic gradient noise as in the non-distributed setups.  The main idea is to apply a \textit{smooth} clipping operator on some estimation errors of local gradients, rather than directly to the local stochastic gradients themselves, followed by the utilization of an \textit{error feedback} mechanism to update the local gradient estimator. We design a time-varying smooth clipping operator $\Psi_t$ (See Figure \ref{fig:sc-operator} for an illustration), parameterized by $\varphi_t, \epsilon_t > 0$, that transforms any scalar input $y \in \R$ as
\begin{align}
     \Psi_t(y) & = \frac{y \varphi_t }{\sqrt{y^2 + \epsilon_t}}. \label{eq:psit} 
\end{align}
Applying $\Psi_t$ on each component of $\y = [y_1, \ldots, y_d]^\top \in \R^d$, we have the component-wise clipping operator on vector $\y$, $\bPsi_t(\y) = [\Psi_t(y_1), \ldots, \Psi_t(y_d)]^\top$.
Then, we use $\m_i^t$ to track the local stochastic gradient and $\m_i^t$ follows from a \textit{weighted error feedback} with \textit{smoothed clipping}, with some weight $0 < \beta_t < 1$, 
\begin{align}
    \m_i^{t+1} & = \beta_t\m_i^t + (1-\beta_t) \bPsi_t(g_i(\x^t) - \m_i^t). \label{eq:mit_update} 
\end{align}
Then, the server computes $\x^{t + 1}$ by combining all local updates, with step size $\eta_t > 0$, 
\begin{align} 
    \x^{t+1} & = \x^t - \frac{\eta_t}{n} \sum_{i=1}^n  \m_i^{t+1}. \label{eq:main_update} 
\end{align}
We refer to the procedures \eqref{eq:psit}-\eqref{eq:main_update} as $\sce$, and $\sce$ is parameterized by smooth clipping parameters $\varphi_t, \epsilon_t$, error-feedback coefficient $\beta_t$, and step size $\eta_t$. We summarize $\sce$ in Algorithm \ref{alg:sce}, and its direct extension $\sen$ in decentralized network setups in Algorithm \ref{alg:sen}.


\begin{algorithm}[t]
\begin{algorithmic}[1] 
\caption{SClip-EF}
\Require $\beta_t, \eta_t, \epsilon_t, \varphi_t, \x^0$, oracle function $g_i$. 
\State $\m_i^0 \leftarrow \zero, \forall  i = 1, \ldots, n$
 \For{$t = 0,\ldots, T-1$}
    \For{node $i = 1, \ldots, n$ in parallel}
    \State $\m_i^{t+1} \leftarrow \beta_t\m_i^t + (1-\beta_t) \Psi_t(g_i(\x^t) - \m_i^t),$ 
    \EndFor
    \State $\x^{t+1} \leftarrow \x^t - \frac{\eta_t}{n} \sum_{i=1}^n  \m_i^{t+1}$
\EndFor
\State \Return $\x^T$
\label{alg:sce}
\end{algorithmic}
\end{algorithm}

\begin{algorithm}[t]
\begin{algorithmic}[1] 
\caption{$\sen$}
\Require $\beta_t, \eta_t, \varphi_t, \epsilon_t, \x_i^0 = \x_j^0, \ \forall i \ne j$. 
\State $\m_i^0 \leftarrow \zero, \forall  i = 1, \ldots, n$
 \For{$t = 0,\ldots, T-1$}
    \For{node $i = 1, \ldots, n$ in parallel} 
    \State $\m_i^{t+1} \leftarrow \beta_t \m_i^t + (1 - \beta_t) \bPsi_t (g_i(\x_i^t) - \m_i^t) \big)$ 
    \State Send $\x_i^t - \eta_t \m_i^{t+1}$ to all neighbors of agent $i$
    \State $\x^{t+1}_i \leftarrow \sum_{j = 1}^n w_{ij} ( \x_j^t - \eta_t \m_j^{t+1})$
    \EndFor 
\EndFor
\State \Return $\{\x_i^T\}$
\label{alg:sen}
\end{algorithmic}
\end{algorithm}

\section{Main results} 
\label{sec:results}
We present our main results in the decentralized network setup, in which the choice of non-negative $\bW = \{ w_{ij} \} \in \R^{n \times n}$ is crucial for consensus. We assume that the communication network admits a mixing matrix $\bW$ such that the following condition holds. 
\begin{assumption} 
\label{as:mixmatrix}
The network $G = (\cV, \cE)$ is directed and strongly connected. The nonnegative weight matrix $\bW \in \R_{+}^{n \times n}$ associated with the communication network in use satisfies that its components $w_{ij} \neq 0$ only if there is a communication link between agent $i$ and $j$, or $i=j$. Further, $\bW$ is doubly stochastic, i.e., $\bW \one = \one$ and $\one^\top \bW = \one^\top$. The preceding conditions guarantee that $\|\bW -  \one \one^\top/n \|_2 := \lambda \in [0, 1)$.  
\end{assumption}

Assumption \ref{as:mixmatrix} is applicable not only to undirected networks but also to the broader class of strongly connected directed networks that possess doubly stochastic weights. This assumption ensures that the second-largest singular value $\lambda$ of matrix $\bW$ remains less than 1, expressed as $\lambda := \|\bW -  \one \one^\top/n \|_2 < 1$ \cite{horn2012matrix}.

\begin{theorem} 
\label{thm:bdd_alpha} 
Suppose Assumptions \ref{as:func}, \ref{as:bdd_hete}, \ref{as:ht-noise}, \ref{as:mixmatrix} hold. Let $\kappa := L/\mu$. Take clipping parameters $\varphi_t, \epsilon_t$, error feedback wight $\beta_t$, and step size $\eta_t$ as follows: 
\begin{align}
\label{eq:sce-hypara}
    \varphi_t = \frac{\cp}{\sqrt{t + 1}}, \ \epsilon_t = \tau (t + 1)^{3/5}, \ \beta_t = \frac{c_\beta}{ \sqrt{t + 1}}, \ \eta_t = \frac{c_\eta}{(t + 1)^{1/5}},  
\end{align}
where constants $\cp, \tau, c_\beta, c_\eta$ satisfy that $c_\eta^2 \cp^2 = \Theta(\sigma^2/\mu^2)$, $0 \le c_\beta < 1$, and $\tau = \Theta(d \sigma^2 \kappa^{5/2})$.
Then, there exists some constant $c_s = \Omega \big(d^{-1/2 } \kappa^{-5/4 }\big)$ such that, for each $i \in [n]$, the iterates $\{ \x_i^t \}$ generated by Algorithm  \ref{alg:sen} satisfy that for any $0 < \delta < (c_s, 2/5)$, it holds that 
\begin{align}
\label{eq:thmcvg}
     \lim_{t \rightarrow \infty}  (t+1)^{ \delta } \E \big[ \| \x_i^t - \x^* \|^2 \big] = 0. 
\end{align}
\end{theorem} 
\begin{remark} 
The theorem above establishes an MSE convergence rate for the proposed decentralized method, Algorithm \ref{alg:sen}, under symmetric heavy-tailed noise with only a bounded first absolute moment. To the best of our knowledge, this is the first convergence rate for general decentralized optimization under heavy-tailed noise and the first result without assuming the bounded gradient condition \cite{sun2024distributed}. To ensure convergence, the algorithm parameters are configured as in \eqref{eq:para-setup}, based on certain analytical constants defined in the analysis. The specific parameter choices in the theorem ensure the lower bound of $c_s$. 
\end{remark}

\begin{corollary}
\label{cor:ser-cli}
If we take $\bW = (1/n)\one\one^\top$, and let $\x^t = \x_i^t, \forall i \in [n]$, then the update of $\x^t$ is equivalent to Algorithm \ref{alg:sce} in the server-client setup. Therefore, under Assumptions \ref{as:func}, \ref{as:bdd_hete}, \ref{as:ht-noise} and the same parameter conditions as in Theorem \ref{thm:bdd_alpha}, the iterates $\{\x^t\}$ from Algorithm \ref{alg:sce} also satisfy that for any $0 < \delta < (c_s, 2/5)$, there holds $\lim_{t \rightarrow \infty} (t + 1)^{\delta} \E[ \|\x^t - \x^*\|^2] = 0$. 
\end{corollary}

\begin{remark}
We again note that $c_s$ is independent of $\alpha$ in the $\alpha$-moment bound assumption made in the related work \cite{yang2022taming}, and the implied Corollary \ref{cor:ser-cli} in server-client setups also does not rely bounded domain or gradient, which we attribute to the power of symmetry in the noise. 
\end{remark}

\section{Convergence analysis}  
\label{sec:proof}
\subsection{Preliminaries} 
Define the network average as $\ox^t := (1/n)\sum_{i=1}^n \x_i^t$. Then, from line 6 in Algorithm \ref{alg:sen}, and the double stochasticity of $\bW$ in Assumption \ref{as:mixmatrix}, the iterative process of the network average $\ox^t$ is 
\begin{align}
\label{eq:x_avg_update}
    \ox^{t+1} = \ox^{t} - \frac{\eta_t}{n}\sum_{i=1}^n \m_i^{t+1}. 
\end{align} 
Define the stacked long vectors $ \x^t = 
    \begin{bmatrix}
    \x_1^t, 
    \ldots,
    \x_n^t
    \end{bmatrix}^\top
    , \quad 
    \m^t = 
    \begin{bmatrix}
    \m_1^t,
    \ldots, 
    \m_n^t
    \end{bmatrix}^\top.$
Then, we can summarize all local updates at iteration $t$ into the compact matrix form,
\begin{align}
\label{eq:alg-stack}
    \x^{t+1} = (\bW \otimes \I_d)(\x^t - \eta_t \m^{t + 1}).
\end{align} 
Let $\x^* = \argmin_{\x \in \R^d} f(\x)$ and $\e_{x}^t  := \bar{\x}^t - \x^*$. Since each $f_i$ is twice-continuously differentiable by Assumption \ref{as:func}, from the Lagrange mean value theorem, we can define $\bH_{f, i}$ for each $i \in [n]$ as in the following, 
\begin{align}
\begin{split}
\label{eq:hf_def} 
    \nabla f_i(\ox^t) - \nabla f_i(\x^*) = \underbrace{\int_{0}^1 \big[\nabla^2 f_i (\x^* + s(\ox^t - \x^*) \big] ds}_{:= \bH_{f, i}^t} (\ox^t - \x^*) =  \bH_{f, i}^t \e_x^t. 
\end{split}
\end{align} 
Recall that we use the function $p$ to denote the density of gradient noise. We define the following expectation over gradient noise,
\begin{align}
\label{eq:phimean}
    \Phi_t(w) := \int_{-\infty}^{\infty} \Psi_t(w + u) p(u) du, \ \forall w \in \R. 
\end{align}
For column vector $\w = [w_1, \ldots, w_d]^\top \in \R^d$, we define the column vector $\bPhi_t(\w) := [\Phi_t(w_1), \ldots, \Phi_t(w_d)]^\top$, 
and the random vector for each $i \in [n]$, 
\begin{align}
\label{eq:eta_i_def}
    \bet_i^t := \bPsi_t(\nabla f_i(\x_i^t) + \bxi^t_i - \m_i^t) - \bPhi_t(\nabla f_i(\x_i^t) - \m_i^t).
\end{align} 
We define a diagonal matrix $\bH_{\Phi, i}^t$ to \textit{linerize} nonlinear operators $\bPhi_t$, i.e., $\forall i \in [n]$,
\begin{align} 
\label{eq:hpsi_i_def}
\begin{split}
     \bPhi_t( \gd f_i(\x_i^t) - \m_i^t)  = \bH_{\Phi, i}^t(\gd f_i(\x_i^t) - \m_i^t). 
\end{split} 
\end{align}
Note that if the $k$-th component $ [\gd f_i(\x) - \m_i^t]_k$ is zero, then the $k$-th diagonal entry of $\bH_{\Phi, i}^t$ can be defined as any quantity. Following from the main update \eqref{eq:main_update}, and using  the relation that $\sum_{i=1}^n \nabla f_i(\x^*) = 0$, \eqref{eq:hf_def}-\eqref{eq:eta_i_def}, one has the following main recursion for the error term $\e_x^t$:
\begin{align*}
\begin{split}
\label{eq:ex_rec_dist_sto} 
     & \ox^{t+1} - \x^*  \\ 
    = \ & \e_x^t - \frac{\eta_t}{n}\sum_{i=1}^n\Big[ \beta_t \m_i^t + (1 - \beta_t) \cdot  \\
    & \qquad \Big( \bH_{\Phi, i}^t \Big( \nabla f_i(\x_i^t) - \nabla f_i(\ox^t) + \nabla f_i(\ox^t) - \nabla f_i(\x^*)  - (\m_i^t - \gd f_i(\x^*))\Big) + \bet_i^t \Big)   \Big]  \\ 
    = \ & \Big[\bI - \frac{\eta_t (1-\beta_t)}{n}\sum_{i=1}^n \bH_{\Phi, i}^t \bH_{f, i}^t\Big] \e_{x}^t  - \frac{\eta_t \beta_t}{n} \sum_{i=1}^n \m_i^t \\
    \ & \qquad + \frac{\eta_t (1-\beta_t)}{n} \sum_{i=1}^n \Big( \bH_{\Phi, i}^t \Big(\nabla f_i(\ox^t) - \nabla f_i(\x_i^t) + \m_i^t - \nabla f_i(\x^*)\Big)  - \bet_i^t \Big). 
\end{split}
\end{align*}
We define the following, 
\begin{align}
& \bu_t := \Big[ \bI - \frac{\eta_t (1-\beta_t)}{n}\sum_{i=1}^n \bH_{\Phi, i}^t \bH_{f, i}^t\Big] \e_{x}^t, \\
& \bv_t := - \frac{\eta_t \beta_t}{n} \sum_{i=1}^n \m_i^t,  \\
& \z_t := \frac{\eta_t (1-\beta_t)}{n} \sum_{i=1}^n \Big( \bH_{\Phi, i}^t \Big(\nabla f_i(\ox^t) - \nabla f_i(\x_i^t) + \m_i^t - \nabla f_i(\x^*)\Big)  \Big),  \label{eq:ztdef} \\
& \br_t := -\frac{\eta_t (1-\beta_t)}{n}\sum_{i=1}^n \bet_i^t. \label{eq:rtdef}
\end{align} 
Then we can simplify the main recursion as
\begin{align}
\label{eq:update_abstract}
\e_x^{t+1} 
& = \bu_t + \bv_t + \z_t + \br_t. 
\end{align}   

\subsection{Intermediate lemmas}
We first show that at each node $i \in [n]$ the local gradient estimator $\m_i^t$ is upper bounded by a decreasing sequence. 
\begin{lemma} 
\label{lm:mt_bound_i} 
For each $i \in [n]$, $\forall t \ge 0$, $|\m_i^{t + 1}| \le 2 \cp / \sqrt{t+ 1}$. 
\end{lemma} 
\begin{proof}
From the step 4 in Algorithm \ref{alg:sen}, and the fact that $|\bPsi_t(g_i(\x_i^t) - \m_i^t)| \le \varphi_t$, we have that for each $i \in [n]$, 
\begin{align}
    |\m_i^{t+1}| \le \beta_t |\m_i^t| +  (1-\beta_t) \varphi_t 
    \label{eq:m_update}, 
\end{align} 
which leads to $|\m_i^{t+1}| \le \max \big(\|\m_i^t\|_\infty, \ \varphi_t \big)$.  Since $\varphi_t = \cp/\sqrt{t + 1}$ is decreasing in $t$, and $\m_i^0 = \zero$, we have $\forall t \ge 0, |\m_i^{t}| \le \cp$. Then substituting the preceding inequality along with $\beta_t = \cb/\sqrt{t + 1}$ into \eqref{eq:m_update}, we obtain
\begin{align*}
    |\m_i^{t+1}| \le \frac{ \cb \cp }{\sqrt{t + 1}} + \frac{\cp(1 - \beta_t)}{\sqrt{t + 1}} = \frac{2\cp}{\sqrt{t + 1}}.
\end{align*}
\end{proof} 

We next bound the consensus error $\|\x_i^t - \ox^t\|$ over the network. 
\begin{lemma}
\label{lm:consensus} 
The iterates $\{\x^t\}_{t \ge 0}$ generated by Algorithm \ref{alg:sen} satisfy that for some constant $c_{\lambda} $, that is dependent on $\lambda$ in Assumption \ref{as:mixmatrix}, such that $\forall t \ge 0$,
\begin{align}
\label{eq:consensus-error-1}
    \|\x^t - \one_n \otimes \ox^t \| \le \frac{2  \cp c_\eta c_{\lambda}  \sqrt{dn}}{(t + 1)^{7/10}}.
\end{align}
\end{lemma}
\begin{proof}
From the global update rule \eqref{eq:alg-stack} we have $\forall t \ge 0$, 
\begin{align}
\label{eq:xt}
    \x^{t + 1 } = (\bW \otimes \I_d)^{t + 1} \x^0 -  \sum_{k = 0}^{t }  (\bW \otimes \I_d)^{t + 1 -k} \eta_k \m^{k + 1}.
\end{align}
Recall that $\bW$ is doubly stochastic, which yields that 
\begin{align}
\begin{split}
\label{eq:ij_sto}
    & \big( \I_{nd} - \frac{1}{n}\mathbf{1}_n \mathbf{1}_n^{\top} \otimes \I_d \big) (\bW \otimes \I_d)\\
    = \ & \big( \bW \otimes \I_d  - \frac{1}{n} \one_n \one_n^\top \otimes \I_d \big) \big( \I_{nd} - \frac{1}{n}\mathbf{1}_n \mathbf{1}_n^{\top} \otimes \I_d \big). 
\end{split} 
\end{align}
From \eqref{eq:ij_sto} and the initialization condition $\x_i^0 = \x_j^0, \forall i, j \in [n]$, we have $\big( \I_{nd} - (1/n) \mathbf{1}_n \mathbf{1}_n^{\top} \otimes \I_d \big) (\bW \otimes \I_d)^{t + 1} \x^0 = \zero$,
and
\begin{align}
\label{eq:cns_main}
\begin{split}
& \|\x^{t + 1} - \mathbf{1}_n \otimes \ox^{t + 1}\| \\
= \ &  \norm{ \big( \I_{nd} - \frac{1}{n} \mathbf{1}_n \mathbf{1}_n^\top \otimes \I_d \big)\sum_{k = 0}^{t}(\bW \otimes \I_d)^{t + 1 -k} \eta_k  \m^{ k + 1 }} \\
\le \ &  \sum_{k=0}^{t} \norm{ \bW -  \frac{1}{n}\mathbf{1}_n \mathbf{1}_n^\top }_2^{t + 1 - k} \norm{ \eta_k \m^{k + 1 } } \\
\le \  &  2  \cp \cet  \sqrt{d n }  \sum_{k=0}^{t} \frac{\lambda^{t + 1 -k}}{(k + 1)^{7/10}}, 
\end{split}
\end{align} 
where in the last step we use $\| \m_i^{t + 1} \| \le 2 \cp  \sqrt{d} /\sqrt{t + 1}$ from Lemma \ref{lm:mt_bound_i}. The upper bound in \eqref{eq:cns_main} falls into the pursuit of Lemma \ref{lm:cvg_series}, i.e., there exists constant $c_{\lambda} $, that is dependent on $\lambda$ in Assumption \ref{as:mixmatrix}, such that $\forall t \ge 1$, \eqref{eq:consensus-error-1} holds true. As $\x^t - \one_n \otimes \ox^t = 0$ when $t = 0$, the lemma is thereby proved.
\end{proof}

We next show that the virtual iterates $\{\ox^t\}$ has bounded growth. 
\begin{lemma}
\label{lm:x_bound_sto}
For any $t \ge 0$, $|\ox^t - \ox^0| \le (20/3) \cet \cp  t^{3/10}$. 
\end{lemma}

\begin{proof}
From the update \eqref{eq:x_avg_update}, for $t \ge 0$, $\ox^{t+1} = \ox^0 - (1/n) \sum_{k = 0}^t \sum_{i=1}^n \eta_k  \m_i^{k+1}$. 
Then, by Lemma \ref{lm:mt_bound_i} and triangle inequality, it follows that 
\begin{align}
\begin{split}
    |\ox^{t+1} - \ox^0| \le \sum_{k=0}^t \frac{2c_\eta \cp}{(k+1)^{7/10}}  \le \int_0^{t+1} \frac{2c_\eta \cp}{ s^{7/10}} ds = \frac{20 c_\eta \cp}{3} (t+1)^{3/10}. 
\end{split}
 \end{align}
\end{proof}

\begin{lemma}
\label{lm:zetak_ub_dist}
Define $R := \| \ox^0 - \x^*\|$. For each $i \in [n]$, and $\forall t \ge 0$, we have
\begin{align*}
     |\nabla f_i(\x_i^t) - \m_i^t|
    \le  \frac{20 \cet \cp \sqrt{d} L}{3} t^{3/10}  + LR + c_* + \frac{2 \sqrt{2} \cp}{\sqrt{t + 1}} + \frac{2 \cp c_\eta c_{\lambda}  \sqrt{dn} L}{(t + 1)^{7/10}}. 
\end{align*}

\end{lemma} 

\begin{proof}
With the $L$-smoothness in Assumption \ref{as:func}, for each $i \in [n]$, 
\begin{align}
\begin{split}
\label{eq:zeta_upperbound}
    & |\nabla f_i(\x_i^t) - \m_i^t|  \\
    \le \ & \|\nabla f_i(\x_i^t) - \nabla f_i(\ox^t)\| + \|\nabla f_i(\ox^t) - \nabla f_i(\ox^0)\| + \| \nabla f_i(\ox^0) - \nabla f_i(\x^*)\| \\
    \ & \qquad + |\nabla f_i(\x^*)| + |\m_i^t| \\ 
    \le \ & L \|\x_i^t - \ox^t\| + L \|\ox^t - \ox^0\| + L \|\ox^0 - \x^* \| + c_* + 2\sqrt{2} \cp/\sqrt{t + 1}.
\end{split}
\end{align}
Applying the upper bounds of $\|\x_i^t - \ox^t\|$ and $\|\ox^t - \ox^0\|$ from Lemma \ref{lm:consensus} and Lemma \ref{lm:x_bound_sto}, respectively, gives the required result. 
\end{proof}

\begin{lemma}
\label{lm:findcb} 
Suppose that Assumption \ref{as:ht-noise} holds true. Then, for any $a > 0$, we have $\int_{-a}^{a} p(u) du \ge 1 - \sigma/a$. For any $0 < c_b < 1$, we have $\int_{-b}^b p(u) du \ge c_b$ for $b = \sigma /(1 - c_b)$.
\end{lemma}

\begin{proof}
For any constant $a > 0$, 
\begin{align*}
    a \big( \int_{-a}^{-\infty} p(u)  du  + \int_{a}^\infty p(u) du \big) \le \int_{-a}^{-\infty} |u| p(u) du + \int_{a}^\infty |u| p(u) du \le \sigma.
\end{align*}
It follows that $\int_{a}^\infty p(u) du \le  \sigma/(2a) \text{ and }
    \int_{-a}^{a} p(u) du \ge 1 - \sigma/a.$
Let $b = \sigma/(1 - c_b)$, then $\int_{-b}^{b} p(u) du \ge 1 - \sigma/b = c_b$.
\end{proof}

\begin{lemma}
\label{lm:hphi_bound_i}
For any fixed constants $0 < c_b < 1, \ 0 < \varepsilon < 1$. Let
\begin{align*}
    c_1 := \frac{(1 - \varepsilon) c_b\cp}{\sqrt{ 98 \cet^2 \cp^2 d L^2 + 2\sigma^2/(1 - c_b)^2 +  \tau  }}, \quad  
    c_2 := \frac{(1 + \varepsilon) \cp}{\sqrt{\tau}}. 
\end{align*}
Then, there exists some finite $t_0$, such that for each $i \in [n]$, $\forall t \ge t_0$, all the diagonal entries of $\bH_{\Phi, i}^t$ fall into the interval $\big[ c_1/(t+1)^{4/5}, c_2/(t+1)^{4/5} \big]$.   
\end{lemma}

\begin{proof}
For simplicity, for each component $k = 1, \ldots, d$, let $\zeta_{k, i}^t := [\nabla f_i(\x_i^t) - \m_i^t]_k$. 
Suppose $\zeta_{k, i}^t \ne 0$. By the definition of $\bH_{\Phi, i}^t$ in \eqref{eq:hpsi_i_def}, its $k$-th diagonal $[\bH_{\Phi, i}^t]_{kk}$ is  
\begin{align}
\begin{split}
\label{eq:HPhikk}
     &  \frac{1}{ \zeta_{k, i}^t} \int_{-\infty}^{\infty}  \frac{\cp}{\sqrt{t + 1 }} \frac{\zeta_{k, i}^t + u}{\sqrt{\big(\zeta_{k, i}^t + u\big)^2 + \epsilon_t}} p(u) du \\
    = \ &  \underbrace{ \int_{-\infty}^{\infty}  \frac{\cp}{\sqrt{t + 1 }} \frac{p(u)}{\sqrt{\big(\zeta_{k, i}^t+ u\big)^2 + \epsilon_t}} du }_{\mathsf{S}_{1, k}^t } +  \underbrace{ \frac{1}{ \zeta_{k, i}^t } \int_{-\infty}^{\infty}  \frac{\cp}{\sqrt{t + 1 }} \frac{ u p(u)}{\sqrt{\big(\zeta_{k, i}^t+ u\big)^2 + \epsilon_t}}  du }_{\msf{S}_{2, k}^t}.
\end{split}
\end{align}
If $\zeta_{k, i}^t = 0$, by the definition of $\bH_{\Phi, i}^t$, we can define $[\bH_\Phi^t]_{kk}$ as any quantity, and thus any bounds on \eqref{eq:HPhikk} trivially follows. We bound the two summands in the last line of \eqref{eq:HPhikk} one by one. 
For the first summand $\mathsf{S}_{1, k}^t$ in \eqref{eq:HPhikk}, we have the upper bound 
\begin{align}
\label{eq:hphi_1_ub}
\begin{split}
    \mathsf{S}_{1, k}^t
    \le   \frac{\cp}{\sqrt{t + 1 }} \frac{1}{\sqrt{\epsilon_t}} = \frac{\cp}{\sqrt{\tau}(t+1)^{4/5}}.
\end{split}
\end{align}
From Lemma \ref{lm:zetak_ub_dist}, 
one can take $t_{0, 1}$ as the minimum $t$ so that $| \zeta_{i, k} | \le 7 \cet \cp \sqrt{d} L t^{3/10}$. 
Take $0 < c_b < 1$ and $b = \sigma/(1 - c_b)$. From Lemma \ref{lm:findcb}, for $t \ge t_{0, 1}$, one has
\begin{align}
\label{eq:hphi_1_lb}
\begin{split}
    \mathsf{S}_{1, k}^t \ge & \  \frac{\cp}{\sqrt{t + 1 }} \int_{-b}^{b} \frac{p(u)}{\sqrt{2(\zeta_{k, i}^t)^2 + 2u^2 + \epsilon_t}} du \\
    \ge & \ \frac{\cp}{\sqrt{t + 1 }} \frac{ c_b}{\sqrt{2(\zeta_{k, i}^t)^2 + 2b^2 + \epsilon_t}} \\
    \ge & \ \frac{c_b\cp}{\sqrt{ 98 \cet^2 \cp^2 d L^2 + 2\sigma^2/(1 - c_b)^2 +  \tau}} \frac{1}{(t+1)^{4/5}}.
\end{split}
\end{align}
We next proceed to bound the second summand $\mathsf{S}_{2, k}^t$ in \eqref{eq:HPhikk}. Define that
\begin{align}
\label{eq:pqdecom}
    P_{k, i}^t := \frac{1}{\sqrt{(\zeta_{k, i}^t - u)^2 + \epsilon_t}}, \ Q_{k, i}^t := \frac{1}{\sqrt{(\zeta_{k, i}^t + u)^2 + \epsilon_t}}.
\end{align}
By the symmetry of density function $p$, we have that 
\begin{align}
\label{eq:pq1}
    \int_{-\infty}^{\infty} \frac{ u p(u)}{\sqrt{\big(\zeta_{k, i}^t+ u\big)^2 + \epsilon_t}} du = \int_0^\infty (Q_{k, i}^t-P_{k, i}^t) u p(u) du.  
\end{align}
We only need to bound $|Q_{k, i}^t - P_{k, i}^t|$ for the case where $u > 0$.  If $\zeta_{k, i} > 0$, we have
\begin{align}
\begin{split}
\label{eq:pq2}
     |Q_{k, i}^t - P_{k, i}^t| 
    = & \ \frac{|(Q_{k, i}^t)^2 - (P_{k, i}^t)^2|}{Q_{k, i}^t + P_{k, i}^t} \\
    \le & \ \frac{4 |u| |\zeta_{k, i}^t|}{[(\zeta_{k, i}^t - u)^2 + \epsilon_t][(\zeta_{k, i}^t + u)^2 + \epsilon_t]} \frac{1}{Q_{k, i}^t} \\
    = & \ \frac{4 |u| |\zeta_{k, i}^t|}{[(\zeta_{k, i}^t - u)^2 + \epsilon_t]\sqrt{(\zeta_{k, i}^t + u)^2 + \epsilon_t}} \\
    \le & \ \frac{4 |u| |\zeta_{k, i}^t|}{\epsilon_t |u| } = \frac{4 |\zeta_{k, i}^t|}{\epsilon_t}. 
\end{split}
\end{align}
If $\zeta_{k, i}^t < 0$, in similar way one can show $|Q_{k, i}^t - P_{k, i}^t| \le 4|\zeta_{k,i}^t|/\epsilon_t$. 
Thus, 
\begin{align}
\begin{split}
\label{eq:hphi_2nd_bound}
  \mathsf{S}_{2, k}^t
  \le \frac{\cp}{|\zeta_{k, i}^t| \sqrt{t + 1} }\int_{0}^\infty |Q_{k, i}^t - P_{k, i}^t| |u| p(u) du  \le \frac{2 \cp \sigma}{\tau (t+1)^{1.1}}, 
\end{split}
\end{align}
where we use \eqref{eq:pq1}-\eqref{eq:pq2} and the moment bound in Assumption \ref{as:ht-noise}. 
Therefore, from \eqref{eq:hphi_1_lb} and \eqref{eq:hphi_2nd_bound}, we have for $\zeta_{k, i}^t \ne 0$, 
\begin{align}
\begin{split}
\label{eq:hpsik_lb}
    [\bH_{\Phi,i}^t]_{kk} 
    & \ge \frac{c_b\cp}{\sqrt{ 98 \cet^2 \cp^2 d L^2 + 2\sigma^2/(1 - c_b)^2 +  \tau }} \frac{1}{(t+1)^{4/5}} - \frac{2\cp \sigma }{\tau (t+1)^{1.1}}  \\
\end{split}
\end{align} 
Take $t_{0, 2}$ be the minimum $t \ge t_{0, 1}$ such that 
for $t \ge t_{0, 2}$, there exists the lower bound:
\begin{align}
\label{eq:hpsik_lb_ep}
    [\bH_{\Phi,i}^t]_{kk} \ge \frac{1}{(t+1)^{4/5}} \frac{ (1- \varepsilon) c_b\cp}{\sqrt{ 98 \cet^2 \cp^2 d L^2 + 2\sigma^2/(1 - c_b)^2 +  \tau  }} .  
\end{align}
From \eqref{eq:hphi_1_ub} and \eqref{eq:hphi_2nd_bound}, one has
\begin{align}
\label{eq:hpsik_ub}
\begin{split}
    [\bH_{\Phi, i}^t]_{kk} \le 
     \frac{\cp}{\sqrt{\tau}} \frac{1}{(t+1)^{4/5 }} + \frac{2\cp \sigma }{\tau (t+1)^{ 1.1 }} 
\end{split}
\end{align}
Let $t_{0, 3}$ be the minimum $t \ge t_{0, 2}$ such that 
for $t \ge t_{0, 3}$, there exists the upper bound: 
\begin{align}
\label{eq:hpsik_ub_ep}
    [\bH_{\Phi, i}^t]_{kk} \le \frac{ ( 1+ \varepsilon) \cp }{\sqrt{\tau} (t+1)^{4/5}}. 
\end{align}
We define $c_1, c_2$ as the lower and the upper bounds in \eqref{eq:hpsik_lb_ep} and \eqref{eq:hpsik_ub_ep}, respectively. Setting $t_0 := t_{0, 3}$ proves this lemma.
\end{proof}  

\begin{lemma}
\label{lm:cnts_bound}
Let $\bS_t = \bI - \frac{\eta_t (1-\beta_t)}{n}\sum_{i=1}^n \bH_{\Phi, i}^t \bH_{f, i}^t$. Taking proper algorithm parameters, we have, for $c_1, c_2, t_0$ as defined in Lemma \ref{lm:hphi_bound_i}, there holds for all $t \ge t_0$,
\begin{align*}
    \| \bS_t \|_2^2 \le 1 - \frac{c_\eta c_1 \mu}{t+1} + \frac{c_\eta c_1 c_\beta \mu }{(t + 1)^{3/2}} + \frac{c_\eta^2 c_2^2 L^2}{(t+1)^{2}}. 
\end{align*}
\end{lemma} 

\begin{proof}
We fix some constants $0 < \varepsilon < 1, 0 < c_b < 1$ and consider $t \ge t_0$ for $t_0$ as defined in Lemma \ref{lm:hphi_bound_i},
Recall that $\bH_{\Phi, i}^t$ is diagonal and $\bH_{f, i}$ is symmetric. One has
\begin{align*}
    \bS_t^\top\bS_t 
    & = \bI -   \frac{\eta_t (1-\beta_t)}{n} \sum_{i=1}^n  \Big( \bH_{\Phi, i}^t \bH_{f, i}^t +  \bH_{f, i}^t \bH_{\Phi, i}^t \Big) \\
    & \quad + \eta_t^2 (1-\beta_t)^2  \Big( \frac{1}{n}\sum_{i=1}^n \bH_{f, i}^t \bH_{\Phi, i}^t\Big)\Big( \frac{1}{n}\sum_{i=1}^n \bH_{\Phi, i}^t \bH_{f, i}^t \Big). 
\end{align*}

Firstly, from Lemma \ref{lm:hphi_bound_i} and $L$-smoothness from Assumption \ref{as:func}, we have 
\begin{align}
\begin{split}
\label{eq:thrid-eig}
& \ \lambda_{\max} \Big(\Big( \frac{1}{n}\sum_{i=1}^n \bH_{f, i}^t \bH_{\Phi, i}^t\Big)\Big( \frac{1}{n}\sum_{i=1}^n \bH_{\Phi, i}^t \bH_{f, i}^t \Big)\Big) =  \Big\| \frac{1}{n}\sum_{i=1}^n \bH_{\Phi, i}^t \bH_{f, i}^t \Big\|_2^2 \\
\le & \ \Big( \frac{1}{n}\sum_{i=1}^n \| \bH_{\Phi, i}^t \|_2  \| \bH_{f, i}^t \|_2 \Big)^2  \le  \frac{c_2^2 L^2}{(t+1)^{8/5 }}. 
\end{split} 
\end{align}

Secondly, we bound  $\lambda_{\min} \big( (1/n) \sum_{i=1}^n (\bH_{f, i}^t \bH_{\Phi, i}^t + \bH_{\Phi, i}^t \bH_{f,i}^t) \big)$. Using notations from  Lemma \ref{lm:ab+ba}, let $\bA = \bH_{\Phi, i}^t, \bB = \bH_{f, i}^t$ and $\bC = \bA \bB$. Then,
\begin{align}
\label{eq:kakb_bounds}
    1 \le k_a \le \frac{c_2}{c_1}, \ 1 \le k_b \le \kappa, 
\end{align}
where $c_1, c_2$ are defined in Lemma \ref{lm:hphi_bound_i}. We next choose constants $\varepsilon, c_b, \tau$ such that
\begin{align}
\label{eq:pos_right}
    \big(\sqrt{k_a} + 1 \big)^2 - k_b \big(\sqrt{k_a} - 1\big)^2 \ge 2\sqrt{k_a}.  
\end{align}  
Note that when $\kappa = 1$, one has $k_b = 1$ and thus \eqref{eq:pos_right} always holds true. When $\kappa > 1$, to enforce \eqref{eq:pos_right}, it suffices to enforce \eqref{eq:pos_right} when $k_b = \kappa$, i.e., 
\begin{align}
\label{eq:sqrtka_bd}
    \frac{\kappa - \sqrt{2\kappa -1}}{\kappa - 1} \underset{ }{\le} \sqrt{k_a}  \underset{}\le \frac{\kappa + \sqrt{2\kappa -1}}{\kappa - 1}.
\end{align}
For the first part in \eqref{eq:sqrtka_bd}, since $\kappa \ge \sqrt{2\kappa - 1}> 1$, naturally, $ 0 \le \frac{\kappa - \sqrt{2\kappa -1}}{\kappa - 1} < 1 \le \sqrt{k_a}. $
We next ensure the second part of \eqref{eq:sqrtka_bd}. The following inequalities hold: 
\begin{align}
    & \sqrt{k_a} \le  \sqrt{\frac{c_2}{c_1}} 
     = \sqrt{\frac{1 + \varepsilon}{1 - \varepsilon}} \cdot \sqrt{\frac{1}{c_b}} \cdot \Big[ \frac{ 98 \cet^2 \cp^2 d L^2}{\tau} + \frac{2\sigma^2}{\tau (1 - c_b)^2} +  1 \Big]^{1/4}, \label{eq:ka_ub} \\ 
     & \Big(1 +  \frac{1}{\sqrt{ 2(\kappa - 1)}} \Big)^2  \le  \frac{\kappa + \sqrt{2\kappa -1}}{\kappa - 1}.  \label{eq:kappat_lb}
\end{align}
For any constant $\phi > 0$, we define 
\begin{align}
\label{eq:tidkap}
    \kaph := 1 + \min \Big(\frac{1}{\sqrt{2(\kappa - 1)}}, \  \phi \Big) \le 1 +  \frac{1}{\sqrt{ 2(\kappa - 1)}} . 
\end{align}
With \eqref{eq:ka_ub}\eqref{eq:kappat_lb}\eqref{eq:tidkap}, to ensure the second part of \eqref{eq:sqrtka_bd}, it suffices to take
\begin{align}
& \frac{1 + \varepsilon}{1 - \varepsilon} = \kaph^{ 5/2 }, \quad
\frac{1}{c_b}  = \kaph, \\
& \frac{ 98 \cet^2 \cp^2 d L^2}{\tau} + \frac{2\sigma^2}{\tau (1 - c_b)^2} +  1 =  \kaph, \label{eq:para-setup}
\end{align}
which are satisfied by choosing
\begin{align}
\label{eq:algparas}
\begin{split}
    \varepsilon  = \frac{\kaph^{ 5/2 } - 1}{\kaph^{ 5/2 } + 1}, \ c_b  = \frac{1}{\kaph}, \ \tau = \frac{1}{\kaph - 1} \Big[  98 \cet^2 \cp^2 d L^2 + \frac{2\sigma^2}{(1 - c_b)^2}     \Big]. 
\end{split}
\end{align}
We have shown that \eqref{eq:pos_right} can be enforced by using the above parameters. Then, in view of \eqref{eq:cnbd} and by Lemma \ref{lm:hphi_bound_i}, we obtain $\lambda_{\min}(\bH_{f, i}^t \bH_{\Phi, i}^t + \bH_{\Phi, i}^t \bH_{f,i}^t) \ge  \frac{c_1 \mu }{(t+1)^{4/5}}.$
Since $\bH_{f, i}^t \bH_{\Phi, i}^t + \bH_{\Phi, i}^t \bH_{f,i}^t$ is real symmetric, by Weyl's theorem \cite{horn2012matrix}, we have
\begin{align*}
    \ \lambda_{\min}\Big( \frac{1}{n} \sum_{i=1}^n (\bH_{f, i}^t \bH_{\Phi, i}^t + \bH_{\Phi, i}^t \bH_{f,i}^t) \Big) 
    & \ge   \frac{1}{n}\sum_{i=1}^n \lambda_{\min} (\bH_{f, i}^t \bH_{\Phi, i}^t + \bH_{\Phi, i}^t \bH_{f,i}^t) \\
    & \ge  \frac{c_1 \mu }{(t+1)^{4/5}}.
\end{align*}
It follows that 
\begin{align*}
    & \lambda_{\max}\Big(\bI -  \frac{\eta_t (1-\beta_t)}{n} \sum_{i=1}^n (\bH_{f, i}^t \bH_{\Phi, i}^t + \bH_{\Phi, i}^t \bH_{f,i}^t)\Big)  \le 1 - \frac{c_\eta c_1 \mu(1 - \beta_t) }{t+1}.
\end{align*}

Thirdly, we can bound $\lambda_{\max} \big( \bS^\top \bS\big)$. Using Weyl's theorem and \eqref{eq:thrid-eig}, one has
\begin{align}
\begin{aligned}
\label{eq:rrbd}
    \lambda_{\max}(\bS^\top \bS) 
    & \le  \lambda_{\max}\Big(\bI - \frac{\eta_t (1-\beta_t)}{n} \sum_{i=1}^n  ( \bH_{\Phi, i}^t \bH_{f, i}^t + \bH_{f, i}^t \bH_{\Phi, i}^t)\Big) \\
    & \quad + \eta_t^2 (1-\beta_t)^2 \lambda_{\max} \Big (\Big( \frac{1}{n}\sum_{i=1}^n \bH_{f, i}^t \bH_{\Phi, i}^t\Big)\Big( \frac{1}{n}\sum_{i=1}^n \bH_{\Phi, i}^t \bH_{f, i}^t \Big) \Big) \\
    & \le 1 - \frac{c_\eta c_1 \mu(1 - \beta_t) }{t+1} + \frac{c_\eta^2 c_2^2 L^2}{(t+1)^{2}}. 
\end{aligned}
\end{align}
Using the fact that $\| \bS_t \|_2^2 = \lambda_{\max}(\bS_t^\top \bS_t)$ yields the required result.  
\end{proof} 

\begin{lemma}
\label{lm:ex_bound}
There exists some constant $c_e$ such that for all $t$, $\E \| \e_x^t \|^2 \le c_e$.
\end{lemma}  
\begin{proof}
We abstract all the history up to iteration $t-1$, i.e., $\{\x_0, \bxi_i^{t'} \}_{i \in [n] , 0 \le t' \le t-1}$, into a rich enough $\sigma$-algebra $\cF_t$. Then, for each $i \in [n]$, one has $\E\big[ \bet_i^t \mid \cF_t \big] = 0$. From the definition of $\Psi_t$, $\| \bet_i^t\|^2 \le 4 d \cp^2/(t + 1)$ a.s.. Taking the conditional expectation on the square of \eqref{eq:update_abstract}, using preceding relations and Jensen's inequality, we obtain
\begin{align}
\label{eq:uvzdecom}
    \E \big[ \| \e_x^{t+1}\|^2 \mid \cF_t \big] \le \| \bu_t + \bv_t + \z_t \|^2 +  \| \br_t \|^2 \text{ and } \| \br_t \|^2 \le \frac{4d c_\eta^2\cp^2}{(t+1)^{7/5}}. 
\end{align}
By Young's inequality, take any constant $0 < \delta_1 < c_\eta c_1 \mu$, 
\begin{align}
\begin{split}
\label{eq:uvzsum}
\| \bu_t + \bv_t + \z_t \|^2 \le \big(1 + \frac{\delta_1}{t+1}\big)\| \bu_t \|^2 + \big(1 + \frac{t+1}{\delta_1 }\big)(2 \| \bv_t \|^2 + 2 \| \z_t \|^2 ). 
\end{split}
 \end{align}
We consider $t \ge t_0$ for $t_0$ defined in Lemma \ref{lm:hphi_bound_i}. Then, from Lemma \ref{lm:cnts_bound} one can find some minimum $t_1 \ge t_0$ so the following holds: 
\begin{align}
\label{eq:cascs}
\begin{split}
    \big(1 + \frac{\delta_1}{t+1}\big)\|\bu_t\|^2 & \le \big(1 + \frac{\delta_1}{t+1}\big) \| \bS_t \|_2^2 \| \e_x^t \|^2 \\
    & \le \Big( 1 - \frac{c_\eta c_1 \mu - \delta_1}{t + 1} + o\big(\frac{1}{t+1}\big) \Big)\| \e_x^t \|^2 \\
    & \le \Big[ 1 - \frac{c_\eta c_1 \mu - \delta_1}{2(t + 1)} \Big] \|\e_x^t\|^2. 
\end{split}
\end{align}
From Lemma \ref{lm:mt_bound_i} we have $\| \m_i^t \| \le 2 \sqrt{2d}\cp  / \sqrt{t + 1} $, then
\begin{align}
\label{eq:vtcntr}
    2\big(1 + \frac{t+1}{\delta_1} \big) \| \bv_t \|^2 \le \frac{16 (1 + 1/\delta_1) d c_\eta^2 c_\beta^2 \cp^2}{(t+1)^{ 7/5 }}.
\end{align}
By Assumption \ref{as:bdd_hete},
\begin{align}
\begin{split}
\label{eq:emit_bound_4}
\| \m_i^t - \nabla f_i(\x^*)  \| & \le \sqrt{d}\big (\frac{2\sqrt{2}\cp}{\sqrt{t + 1}} + c_* \big). 
\end{split}
\end{align}
Thus, with Lemma \ref{lm:hphi_bound_i} and Lemma \ref{lm:consensus}, we have
\begin{align}
\label{eq:ztcntrbd1} 
2\Big(1 + \frac{t+1}{\delta_1} \Big) \| \z_t  \|^2 \le
\frac{2(1 + 1/\delta_1) c_\eta^2 c_2^2 }{t+1} \Big[ 2d \Big( \frac{16\cp^2}{t + 1} + 2c_*^2 \Big) + \frac{8 d n L^2 \cp^2 \cet^2 c_{\lambda} ^2 }{(t+1)^{ 7/5 }} \Big]. 
\end{align}
Putting the relations \eqref{eq:uvzdecom}-\eqref{eq:ztcntrbd1} together, we obtain that for $t \ge t_1$,
\begin{align*} 
& \ \E \big[ \| \e_x^{t+1}\|^2 \mid \cF_t \big] \\
\le & \  \Big[1 - \frac{c_\eta c_1 \mu - \delta_1}{2(t+1)} \Big] \| \e_x^t \|^2 + \frac{8d(1 + 1/\delta_1) c_\eta^2 c_2^2 c_*^2}{t + 1}  + \frac{16 d(1 + 1/\delta_1) c_\eta^2 c_\beta^2 \cp^2 + 4dc_\eta^2 \cp^2}{(t+1)^{ 7/5 }} 
\\ & \qquad +  \frac{64d(1 + 1/\delta_1) c_\eta^2 c_2^2\cp^2}{(t+1)^2} + \frac{ 16 (1 + 1/\delta_1) d n L^2 c_2^2 \cp^2 c_{\lambda} ^2 c_\eta^4  }{(t + 1)^{12/5}}. 
\end{align*}
Taking the unconditional expectation on both sides of the above relation and applying Lemma \ref{lm:finite_bound} in Appendix leads to that $\E \| \e_x^t \|^2$ is upper bounded by some constant $c_e$.
\end{proof}

We next derive a tighter bound for $\E\|\z_t\|^2$ utilizing the boundedness of $\E \| \e_x^{t} \|^2$. 
\begin{lemma}
\label{lm:ezbound}
There exists some constant $c_z$ such that for all $t$ it satisfies that $ \E \big[ \| \z_t \|^2 \big] \le c_z/(t+1)^{12/5}$. 
\end{lemma} 

\begin{proof}
By definition, 
\begin{align}
\label{eq:ezdecomp}
\E \| \z_t \|^2 = \E\big[ \| \z_t \|^2 \mathds{1}_{\| \e_x^t \| \le (t+1)^{1/5}}   \big] +  \E\big[ \| \z_t \|^2 \mathds{1}_{\| \e_x^t \| > (t+1)^{1/5}} \big].
\end{align} 
We bound the two terms on the right hand side one by one. Without loss of generality, we assume $\zeta_{k, i}^k \ne 0$. Recall that in \eqref{eq:HPhikk}, we let $[\bH_{\Phi, i}^t]_{kk} = \msf{S}_{1, k}^t + \msf{S}_{2, k}^t$. 

\textbf{Step 1}. We bound the first term in \eqref{eq:ezdecomp}. Suppose $\|\e_x^t\| \le (t+ 1)^{1/5}$. Under this condition, we can derive a tighter lower bound for $\msf{S}_{1, k}^t$ than that in \eqref{eq:hphi_1_lb}. From the upper and lower bounds of $\msf{S}_{1, k}^t$ in \eqref{eq:hphi_1_ub} and \eqref{eq:hphi_1_lb}, respectively,  it is clear that  $\msf{S}_{1, k}^t$ deviates from the upper bound $\frac{\cp}{\sqrt{\tau}(t+1)^{4/5}}$ with uncertainty in $\cO(1/(t+1)^{4/5})$. We proceed to reduce this uncertainty. 

From \eqref{eq:zeta_upperbound} we have for some constant $c_3$, 
\begin{align}
\begin{split}
\label{eq:zeta_bd_4}
    |\zeta_{k, i}^t| & \le L(t+1)^{1/5}  + c_* + \frac{2 \sqrt{2} \cp}{\sqrt{t + 1}} + \frac{2 \cp c_\eta c_{\lambda}  \sqrt{dn} L}{(t + 1)^{7/10}} \le c_3 ( t+ 1)^{1/5}. 
\end{split}
\end{align}
It follows that for sufficiently large $t$, 
\begin{align}
\label{eq:s1-tighter-lb}
\begin{split}
\msf{S}_{1, k}^t \ge & \ \int_{-\infty}^{\infty}  \frac{\cp}{ \sqrt{t + 1} } \frac{p(u)}{\sqrt{2c_3^2(t + 1)^{2/5} + 2u^2 + \epsilon_t}} du \\
\ge & \ \int_{-(t+1)^{1/5 }}^{(t+1)^{1/5}}  \frac{\cp}{ \sqrt{t + 1} } \frac{p(u)}{\sqrt{(2c_3^2 + 2) (t+1)^{ 2/5 } + \epsilon_t}} du \\
= & \ \frac{\cp}{\sqrt{\tau} (t+1)^{4/5} } \frac{1}{ \sqrt{1 + \frac{2c_3^2 + 2}{\tau(t+1)^{1/5}}} } \int_{- (t+1)^{1/5}}^{(t+1)^{1/5}} p(u) du \\
\overset{(a)}{\ge} & \ \frac{\cp }{\sqrt{\tau} (t+1)^{4/5}  } \Big[1 - \frac{2c_3^2 + 2}{\tau(t+1)^{1/5}} \Big]\Big[1 - \frac{\sigma}{(t +1)^{1/5}} \Big] \\ 
\ge & \ \frac{\cp }{\sqrt{\tau} (t+1)^{4/5} } \Big[1 - \frac{(2c_3^2 + 2)/\tau + \sigma}{(t+1)^{1/5}} \Big] 
\end{split}
\end{align}
where in step $(a)$ we take sufficiently large $t$ so that $(t+1)^{1/5} > (2c_3^2 +2)/\tau$, and use the fact that for every $0 < a < 1$, $1/\sqrt{1 + a} \ge \sqrt{1 - a} \ge 1 - a$, and also Lemma \ref{lm:findcb}. The faster decaying uncertainty of  $\msf{S}_{1,k}^t$, as shown in \eqref{eq:s1-tighter-lb}, allows us to leverage the useful equality, i.e., $\sum_{i=1}^n \nabla f_i(\x^*) = 0$, when bounding $\E\big[ \|\z_t \|^2 \big]$, as illustrated in the forthcoming analysis. 

Define the following diagonal matrices in $\R^{d \times d}$: 
\begin{align*}
    \bQ^t := \diag\big[ \frac{\cp}{\sqrt{\tau}(t+1)^{4/5}}, \ldots  \big], \ \msf{S}_1^t := \diag\big[ \msf{S}_{1, 1}^t, \ldots, \msf{S}_{1, d}^t ], \ \msf{S}_2^t := \diag\big[ \msf{S}_{2, 1}^t, \ldots, \msf{S}_{2, d}^t ].
\end{align*}
Recall the definition of $\z_t$ in \eqref{eq:ztdef} and use $\sum_{i=1}^n \nabla f_i(\x^*) = 0$, for some constant $c_{z, 0}$, 

\begin{align*}
 & \ \E \big[ \| \z_t \|^2 \mathds{1}_{\| \e_x^t \| \le (t+1)^{1/5}} \big] \\ 
\le & \  2 \Big \| \frac{\eta_t }{n}  \sum_{i=1}^n  \bH_{\Phi, i}^t (\m_i^t  - \nabla f_i(\x^*))  \Big \|^2  + 2 \Big \| \frac{\eta_t }{n}  \sum_{i=1}^n  \bH_{\Phi, i}^t (\nabla f_i(\ox^t) - \nabla f_i(\x_i^t)) \Big \|^2  \\ 
\le & \ 2 \Big\| \frac{\eta_t }{n}\sum_{i=1}^n   \big(\bQ^t - (\bQ^t - \msf{S}_1^t) + \msf{S}_2^t \big) (\m_i^t - \nabla f_i(\x^*))  \Big \|^2 + \frac{8 dnL^2 c_2^2 c_{\lambda} ^2 \cp^2 c_\eta^4}{(t+1)^{3.4}}    \\
\le & \  6 \eta_t^2 \Big[   \Big\| \frac{\bQ^t}{n} \sum_{i=1}^n \m_i^t \Big\|^2  + \Big\| \frac{1}{n} \sum_{i=1}^n  (\bQ^t - \msf{S}_1^t)  (\m_i^t - \nabla f_i(\x^*))  \Big \|^2  \\
& \qquad  + \Big\| \frac{1}{n} \sum_{i=1}^n  \msf{S}_{2}^t (\m_i^t - \nabla f_i(\x^*)) \Big \|^2   \Big] + \frac{8 dnL^2 c_2^2 c_{\lambda} ^2 \cp^2 c_\eta^4}{(t+1)^{3.4}}  \\
\overset{\eqref{eq:s1-tighter-lb}\eqref{eq:hphi_2nd_bound}}{\le} & \ 6\eta_t^2 \Big[ \frac{4 d \cp^4 }{\tau (t+1)^{2.6}  } + \frac{d \cp^2((2c_3^2 + 2)/\tau+\sigma)^2(8 \cp^2/(t + 1) + 2c_*^2)}{\tau^3 (t+1)^{2}}  \\
& \qquad + \frac{4 d\cp^2\sigma^2(8\cp^2/(t + 1) + 2c_*^2)}{\tau^2 (t+1)^{ 2.2 }} \Big] + \frac{8 dnL^2 c_2^2 c_{\lambda} ^2 \cp^2 c_\eta^4}{(t+1)^{3.4}} \\
\le & \ \frac{c_{z, 0}}{(t + 1)^{12/5}}.
\end{align*}

\textbf{Step 2}. We bound the second term in \eqref{eq:ezdecomp}. From Lemma \ref{lm:ex_bound} and using Markov inequality, one has
\begin{align}
\label{eq:markov}
    \P\big( \| \e_x^t \| \ge (t+1)^{1/5}  \big) \le  \frac{ \E \| \e_x^t \|^2 }{(t+1)^{2/5}} \le \frac{c_e}{(t+1)^{2/5}}. 
\end{align} 
Using an upper bound for $\| \z_t \|^2$ independent of $\e_x^t$, similar to \eqref{eq:ztcntrbd1}, it yields that for some constant $c_{z, 1}$, 
\begin{align}
\begin{split}
& \E \big[ \| \z_t \|^2 \mathds{1}_{\| \e_x^t \| > (t+1)^{1/5}} \big] \\
\le \ & 
\frac{ c_\eta^2 c_2^2 }{(t+1)^2} \Big[ 2d \Big( \frac{16\cp^2}{t + 1} + 2c_*^2 \Big) + \frac{8 d n L^2 \cp^2 \cet^2 c_{\lambda} ^2 }{(t+1)^{ 7/5 }} \Big] \P\big( \| \e_x^t \| \ge (t+1)^{1/5}  \big)  \\
\overset{\eqref{eq:markov}}{\le} \ &  \frac{c_{z, 1}}{( t + 1)^{12/5}}.
\end{split}
\end{align}
Combining step 1 and 2 and defining $c_z = c_{z, 0} + c_{z, 1}$ conclude the proof.
\end{proof}

\begin{proof}[Proof of Theorem \ref{thm:bdd_alpha}]
Combining \eqref{eq:uvzdecom}-\eqref{eq:vtcntr} gives that for $t \ge t_0$, 
\begin{align}
\begin{split}
\label{eq:condexp_recursion}
     & \E\big[ \| \e_x^{t+1}  \|^2 \mid \cF_t\big] \\
     \le & \  \Big( 1 - \frac{c_\eta c_1 \mu - \delta_1}{t + 1} + o\big(\frac{1}{t +1}\big) \Big)  \| \e_x^t \|^2  + \frac{16 (1 + 1/\delta_1) d c_\eta^2 c_\beta^2 \cp^2 + 4d c_\eta^2\cp^2 }{(t+1)^{ 7/5 }}  \\
    & \quad  +2\Big(1 + \frac{t+1}{\delta_1} \Big) \| \z_t  \|^2. 
\end{split}
\end{align}
Taking unconditional expectation and using Lemma \ref{lm:ezbound} leads to that there exists some constant $c_4$ such that 
\begin{align} 
\label{eq:main-recursion}
     \E\big[ \| \e_x^{t+1}  \|^2 \big]
     \le   \Big( 1 - \frac{c_\eta c_1 \mu - \delta_1}{t + 1} + o\big(\frac{1}{t+1}\big) \Big)  \E \big[ \| \e_x^t \|^2 \big] + \frac{c_4}{(t + 1)^{7/5}}. 
\end{align}
Take any $0 < \delta_2 <\min(c_\eta c_1 \mu - \delta_1, 2/5)$. Since the function $t \rightarrow (t+2)^{\delta_2}, 0 < \delta_2 \le 1$ is concave in $t$, then $(t + 2)^{\delta_2} \le (t + 1)^{\delta_2}[1 + \delta_2(t+1)^{- 1}]$. In addition, one has $(t + 2)^{\delta_2} = (1 + \frac{1}{t +1})^{\delta_2}(t + 1)^{\delta_2} \le e^{\delta_2/(t + 1)}(t+1)^{\delta_2} \le e(t + 1)^{\delta_2}$. Let $v^t = (t + 1)^{\delta_2} \| \e_x^t \|^2$, using the preceding inequalities on \eqref{eq:condexp_recursion}, we obtain 
\begin{align*}
    \E [ v^{t + 1} ] \le \Big(1 - \frac{c_\eta c_1 \mu - \delta_1 - \delta_2  }{t + 1} + o \big( \frac{1}{t + 1} \big) \Big) \E \big[ v^t \big] + \frac{e c_4}{(t+1)^{7/5 - \delta_2}}. 
\end{align*}
Applying Lemma \ref{lm:cvg_diff_main} on the above relation leads to that 
\begin{align}
\label{eq:net-avg-cvg}
    \lim_{t \rightarrow \infty} (t + 1)^{\delta_2}  \E \big[ \| \e_x^t \|^2 \big] = 0. 
\end{align}
Since $\delta_1$ is arbitrary in $(0, c_\eta c_1 \mu)$, \eqref{eq:net-avg-cvg} holds true for any $\delta_2 < \min(c_\eta c_1\mu, 2/5)$. 
For each $i \in [n]$, by Young's inequality, 
\begin{align*}
    \| \x_i^t - \x^* \|^2 \le  2\| \x_i^t - \ox^t \|^2 + 2 \| \ox^t - \x^* \|^2. 
\end{align*}
From Lemma \ref{lm:consensus}, we have $\| \x_i^t - \ox^t \|^2 = \cO((t+1)^{-7/5})$, which decays faster than $\cO(t^{-2/5})$. Thus, by defining $c_s := c_\eta c_1 \mu$, we have $\forall \delta < \min(c_s, \ 2/5)$, 
\begin{align*}
    \lim_{t \rightarrow \infty} ( t + 1)^\delta \E \big[ \|\x_i^t - \x^* \|^2] = 0. 
\end{align*}

\textit{Lower bounds for $c_s$.} Recall that $c_1$ defined in Lemma \ref{lm:hphi_bound_i} is dependent on the choices of the analytical constants $0 < c_b < 1, \ 0 < \varepsilon < 1$, and algorithm parameters $c_\eta, c_\varphi, \tau$, i.e.,
\begin{align}
\begin{split}
\label{eq:cs_def}
    c_s  = c_\eta c_1 \mu = \frac{(1 - \varepsilon) c_b } {\sqrt{ 98 d \kappa^2 +  \frac{1}{\cp^2 \cet^2 \mu^2 } \big( \frac{2\sigma^2}{(1 - c_b)^2} +  \tau \big)  }}  . 
\end{split}
\end{align}
If $\kappa = 1$, we are free to choose analytical constants $c_b, \varepsilon$, and algorithm parameters. We take $c_b = 9/10 , \ \varepsilon = 1/9, \ c_\eta^2 \cp^2 = 200\sigma^2/\mu^2, \tau = 200 \sigma^2$, then one has $c_s \ge 2/(25\sqrt{d})$ for $\kappa =1 $. 
For the case $\kappa > 1$, we define $\kaph = 1 + \min \big(\phi, \frac{1}{\sqrt{2(\kappa  -1)}} \big)$ for constant $\phi > 0$. Using $\varepsilon, c_b, \tau$ as defined in \eqref{eq:algparas} and substituting them into \eqref{eq:cs_def} lead to 
\begin{align}
\begin{split}
\label{eq:cs_lb_kg1}
    c_s & = \frac{2\sqrt{\kaph - 1}}{ \kaph^{ 3/2 } ( \kaph^{ 5/2 } + 1)  }   \frac{1}{ \sqrt{ 98 d \kappa^2 + \frac{1}{c_\eta^2 \cp^2 \mu^2}  \frac{2\sigma^2 \kaph^2}{(\kaph - 1)^2 }  } } \\
    & \ge \frac{1}{ (1 + \phi)^4}  
   \frac{1}{ \sqrt{ 
   \frac{98 d \kappa^2}{\kaph  - 1 } + \frac{2\sigma^2}{ c_\eta^2 \cp^2 \mu^2 } \frac{\kaph^{2} }{(\kaph - 1)^{3}}}},
\end{split}
\end{align} 
where we used $\kaph \le 1 + \phi$.  We take $\phi = \big(1.25\big)^{ 1/4 } - 1,  \ c_\eta^2 c_\varphi^2 = 3  \sigma^2/\mu^2$, Then, it follows that
\begin{align*}  
    c_s 
    & \ge \frac{4}{ 5 } \frac{1}{ \sqrt{ 
   \frac{98 d \kappa^2}{\kaph  - 1 } + \frac{1}{(\kaph - 1)^{ 3 }}}} \\
 & \ge \begin{cases}
    & \frac{4}{5}\frac{1}{\sqrt{98d \kappa^2 \sqrt{2(\kappa - 1)} + [2(\kappa - 1)]^{ 3/2 } } }, \text{ when } \phi > \frac{1}{\sqrt{2 (\kappa - 1) }}, \text{ i.e., } \kappa \ge 153,   \\
    & \frac{4}{5} \frac{1}{\sqrt{ \frac{98 d\kappa^2}{\phi} + \frac{1}{\phi^{ 3 }} }} \text{ otherwise, }
 \end{cases} \\ 
 & = \Omega \Big( \frac{1}{ d^{ 1/2 } \kappa ^{ 5/4} } \Big). 
\end{align*}  
Further, $\tau = \frac{1}{\kaph - 1} [  98 \cet^2 \cp^2 d L^2 + \frac{2\sigma^2}{(1 - c_b)^2}  ] =  \sigma^2[ \frac{294 d \kappa^2}{\kaph  -1 } + \frac{2 \kaph^2}{(\kaph - 1)^3}] = \Theta(d \sigma^2 \kappa^{ 5/2 })$. Thus, the proof is complete.
\end{proof}  
\section{Numerical experiments}
\label{sec:experiments} 
In the decentralized network setup, we implemented DSGD \cite{nedic2009distributed} (in the form of local update $\x^{t+1}_i = \sum_{j = 1}^n w_{ij} ( \x_j^t - \eta_t g_j(\x_j^t))$), Network-CClip (as in \eqref{eq:network_clippings}), Network-GClip (as in \eqref{eq:network_clippings}), and \scen. For the server-client setup (i.e., the fully connected network), we implemented SGD, Dist-CClip, Dist-GClip, Prox-clipped-SGD-shift \cite{gorbunovhigh}, and \sce. For SClip-EF-Network and \sce, we tuned four hyperparameters: $\cp, \tau, c_\beta$, and $c_\eta$, as described in \eqref{eq:sce-hypara}. For Prox-clipped-SGD-shift, we tuned two constant step sizes and an exponentially decaying (or constant, if more effective) clipping threshold. For all other methods, we adjusted the decay factor $a$ for the decaying step size $a/(t+1)$ and the constant clipping threshold. All algorithms were initialized from $\zero$, fine-tuned using grid search, and the average optimality gap from multiple independent runs is presented.

We solve a synthetic quadratic minimization problem $\min_{\x \in \R^d} \sum_{i=1}^n f_i$, where each $f_i$ is private to node $i$ and defined as $f_i \triangleq (1/2)\x^\top \bA_i \x + \bb_i^{\top} \x$. We set $n = 20$ and $d = 10$. For every $i$, components of $\mathbf{A}_i \in \mathbb{R}^{d \times d}$ and $\mathbf{b}_i \in \mathbb{R}^d$ are uniformly sampled from $[-1, 1]$, and we set $\bA_i$ symmetric and add $0.01$ to its diagonals to ensure positive definiteness. Each node $i$ has access to the stochastic gradients of $f_i$ and gradient noises are i.i.d. sampled from distribution \eqref{eq:ht-example}. We discretize the cumulative distribution function (c.d.f) of \eqref{eq:ht-example} in the truncation range of $[-100, 100]$ and use inverse transform sampling to sample noises. Exact noise sampling from \eqref{eq:ht-example} would not guarantee convergence of other distributed algorithms such as \pcss, since the distribution lacks any $\alpha$ moment for $\alpha > 1$. With truncation, moments of all orders exist, allowing us to compare the convergence rates of these algorithms. We test SClip-EF-Network in a connected cycle of 20 nodes with a degree of 5 and SClip-EF in the server-client setup. As shown in Figure \ref{fig:synthetic_exp}, even with a truncated version of heavy-tailed noise, SClip-EF-Network and SClip-EF converge among the fastest.
\begin{figure}[htbp]
    \centering 
    \begin{subfigure}[b]{0.328\textwidth}
        \centering
        \includegraphics[width=\textwidth]{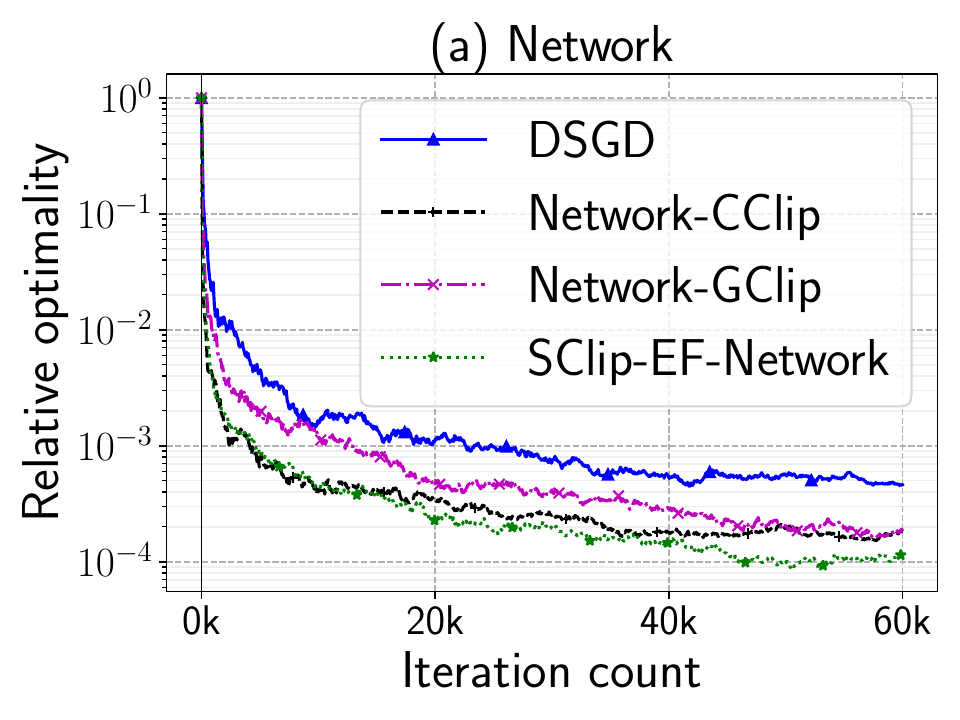}
    \end{subfigure}
    \hfill
    \begin{subfigure}[b]{0.328\textwidth}
        \centering \includegraphics[width=\textwidth]{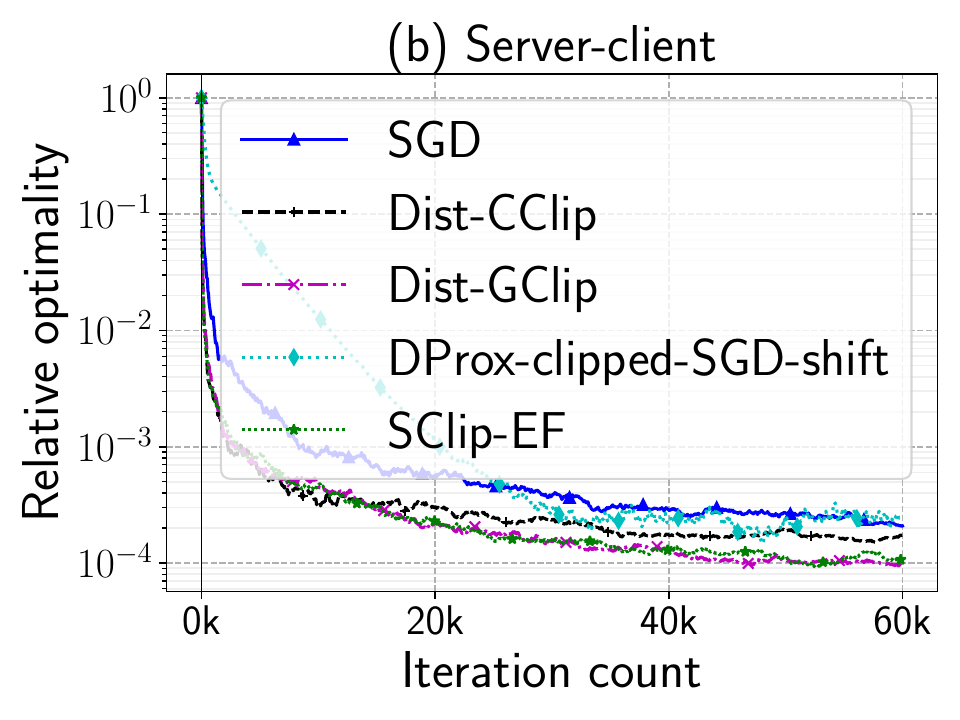}
    \end{subfigure}
    \begin{subfigure}[b]{0.328\textwidth}
    \centering
    \includegraphics[width=\textwidth]{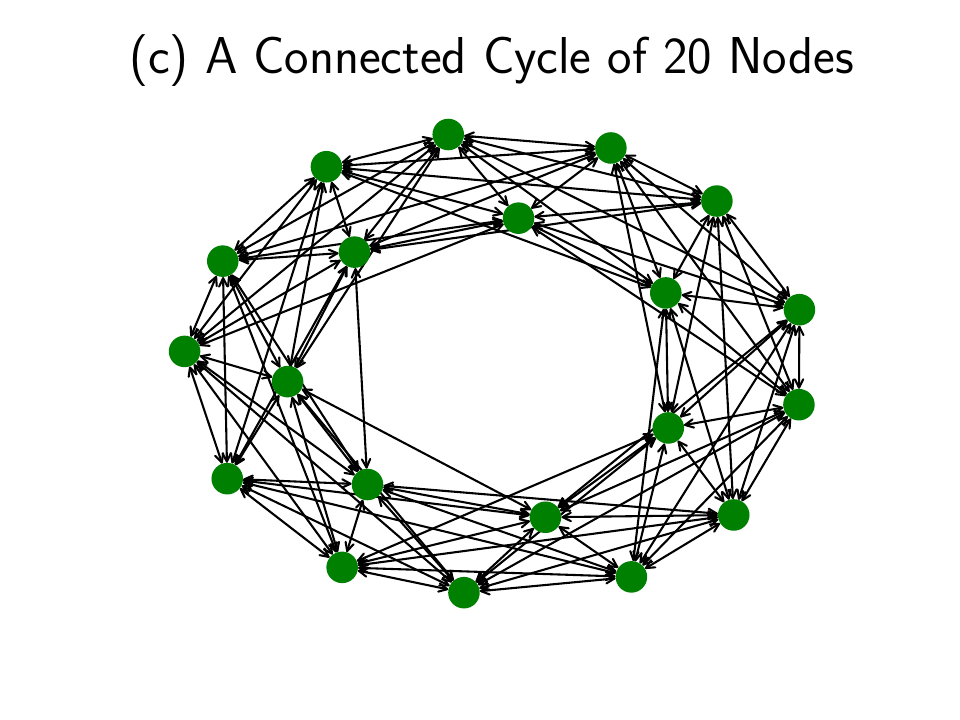}
    \end{subfigure}
    \caption{Average relative optimality $\log_{10} ( f(\x^t) - f(\x^*)) / (f(\x^0) - f(\x^*))$ out of 10 runs in network and server-client cases, and network graph, from left to right.} 
    \label{fig:synthetic_exp}
\end{figure}
\section{Conclusion}
\label{sec:con}
In this work, we have studied strongly convex and smooth heterogeneous decentralized optimization under symmetric heavy-tailed gradient noise. We have proposed SClip-EF and SClip-EF-Network, two distributed methods that integrate a smooth component-wise clipping operator with decaying clipping thresholds and a weighted feedback mechanism to cope with heavy-tailed stochastic gradients and node heterogeneity. The crux of our approach lies in clipping certain gradient estimation errors rather than the stochastic gradients directly.  Under the mild assumption that the gradient noise has only bounded first absolute moment and is (component-wise) symmetric about zero, we have shown that, for the first time in the heterogeneous decentralized case, SClip-EF-Network converges to the optimum with a guaranteed sublinear rate. The effectiveness of the proposed algorithms are supported by numerical examples. Future directions include extending the currently analysis to non-convex functions.  
\section{Appendix}  

\subsection{Proof of Proposition \ref{prop:example}} 
\label{sec:example}
\textbf{Part I}. We first show that it is a \textit{heavy-tailed} distribution. For any positive  $x > 0$, we have $\P(u > x) = \int_x^\infty  \frac{c_p}{(u^2 + 2) \ln^2(u^2 + 2)} du.$
Let $w = \ln (u^2 + 2)$, then $u = \sqrt{e^w - 2}$ and $dw = (2u)du/(u^2 + 2)$. Then, 
\begin{align}
\label{eq:tail-lb}
    \P(u > x) = \int_{\ln(x^2 + 2)}^\infty \frac{c_p}{ 2u w^2} dw = \int_{\ln(x^2 + 2)}^\infty \frac{c_p}{2 w^2\sqrt{e^w - 2}} dw \ge  \int_{\ln(x^2 + 2)}^\infty \frac{c_p}{2 w^2 e^{w/2} } dw. 
\end{align}
We next bound the rightmost term. Consider any constant $a > 0$, 
\begin{align*}
    \int_a^{\infty} \frac{1}{w^2 e^{w/2}} dw = \frac{e^{-a/2}}{a} - \frac{\Gamma(0, \frac{a}{2})}{2}. 
\end{align*}
Using the Theorem 1.1 in \cite{pinelis2020exact}, we have $ \Gamma(0, \frac{a}{2}) < e^{-\frac{a}{2}}\ln\big(1 + \frac{2}{a}\big)$. 
It follows that 
\begin{align*}
    \int_a^{\infty} \frac{1}{w^2 e^{w/2}} dw  >  \frac{1}{2}e^{-\frac{a}{2}}\big[\frac{2}{a} - \ln \big(1 + \frac{2}{a}\big)\big].
\end{align*}
Note that for any $0 < s < \frac{3}{4}$, using Newton-Mercator series we have
\begin{align*}
    s - \ln(1 + s) & = \frac{s^2}{2} - \frac{s^3}{3} + \frac{s^4}{4} - \frac{s^5}{5} + \frac{s^6}{6} - \frac{s^7}{7} \ldots \\
    & =  \frac{s^2}{2}\Big(1 - \frac{2s}{3} \Big) + \frac{s^4}{4}\Big(1 - \frac{4s}{5}\Big) + \frac{s^6}{6}\Big(1 - \frac{6s}{7}\Big) + \ldots \ge \frac{s^2}{4} 
\end{align*}
In the case that $a > \frac{8}{3}$, from the above two relations we have $\int_a^{\infty} \frac{1}{w^2 e^{w/2}} dw > \frac{1}{2 a^2 e^{a/2}}$. 
Combing the above relation with \eqref{eq:tail-lb} gives that when $\ln(x^2 + 2) > \frac{8}{3}$,
\begin{align*}
\P(u > x) > \frac{c_p}{4 \sqrt{x^2 + 2} \ln^2(x^2 + 2) } \ge \frac{c_p}{4\sqrt{x^2 + 2}(x^2 + 1)^2} \ge \frac{c_p}{4}(x^2+2)^{-\frac{5}{2}}. 
\end{align*}
Then it follows that for any $r > 0$, we have 
\begin{align*}
    \lim_{x \rightarrow \infty} e^{rx} \P(u > x) > \lim_{x \rightarrow \infty} \frac{c_p}{4}(x^2+2)^{-\frac{5}{2}}e^{rx} = \infty. 
\end{align*}

\textbf{Part II.} We next show that the distribution as in \eqref{eq:ht-example} has finite first absolute moment. Again, let $w = \ln(u^2 + 2)$, we have 
\begin{align*}
    \int_{-\infty}^{\infty} |u| p(u) du = 2\int_{0}^\infty \frac{uc_p}{(u^2 + 2) \ln^2(u^2 + 2)} du = \int_{\ln 2}^\infty \frac{c_p}{w^2} dw = \frac{c_p}{\ln 2}. 
\end{align*}

\textbf{Part III.} For any $\alpha > 1$, 
\begin{align*}
    \int_{-\infty}^\infty |u|^{\alpha} p(u) du 
    & = 2\int_{0}^\infty \frac{u^{\alpha -1 } u c_p }{(u^2 + 2)\ln^2 (u^2 +2 )} du  = \int_{\ln 2}^{\infty} \frac{(e^w - 2)^{(\alpha - 1)/2} c_p}{w^2} dw \\
    & > \int_{2 \ln 2}^\infty \frac{(\frac{e^{w}}{2})^{(\alpha-1)/2} c_p}{w^2} dw \\
    & > \frac{c_p}{2^{(\alpha-1)/2}} \int_{2 \ln 2}^\infty \frac{1 + w(\alpha -1)/2}{w^2} dw \\
    & > \frac{c_p(\alpha-1)}{2^{(\alpha+1)/2}} \int_{2 \ln 2}^{\infty} \frac{1}{w} dw  = \infty. 
\end{align*}
Thus, any moment $\alpha > 1$ does not exist. 

\subsection{Technical Lemmas}
\begin{lemma} 
\label{lm:cvg_series}
For any constant $\gamma \in (0, 1)$ and $\nu \in (0, 1)$, there exists some constant $c_{\gamma, \nu}$, depending on $\gamma$ and $\nu$, such that for any $t \ge 1$, the following inequality holds:
\begin{align}
\label{eq:cvg_series}
    \sum_{k = 0}^{t-1} \frac{\gamma^{t-k}}{(k + 1)^\nu} \le \frac{c_{\gamma, \nu}}{(t+1)^\nu}. 
\end{align}
\end{lemma}
\begin{proof}
Note that \eqref{eq:cvg_series} is equivalent to that $\forall t \ge 1$, 
\begin{align}
\label{eq:bounded_sum}
    \sum_{k = 1}^{t} \gamma^k \Big(\frac{t+1}{t + 1 - k}\Big)^\nu \le c_{\gamma, \nu}. 
\end{align} 
We focus on the case that $t > \frac{3\nu}{\ln (1/\gamma)} - 1$, and for the finite complementary case, we define
\begin{align*}
    c_0 := \sup_{1 \le t \le \max\big(\lfloor \frac{3\nu}{ \ln (1/\gamma)} - 1 \rfloor, \ 1\big)} \sum_{k = 1}^{t} \gamma^k \Big(\frac{t+1}{t + 1 - k}\Big)^\nu. 
\end{align*}
Now, for $t > \frac{3\nu}{\ln (1/\gamma)} - 1$, we first consider the first half of the summation in \eqref{eq:bounded_sum}, i.e.,  $k = 0, 1, \ldots, \lfloor t/2 \rfloor$. We have 
\begin{align*}
    \Big(\frac{t+1}{t + 1 - k}\Big)^\nu = \exp\Big( - \nu \ln \big(1 - \frac{k}{t+1}\big) \Big). 
\end{align*} 
Using Taylor series, 
\begin{align*}
     -\ln \big(1 - \frac{k}{t+1}\big) 
    \le \ & \frac{k}{t+1} + \frac{1}{2} \Big( \frac{k}{t+1}\Big)\Big[ \frac{k}{t+1} + \Big( \frac{k}{t+1}\Big)^2 + \Big( \frac{k}{t+1}\Big)^3 +   \ldots \Big] \\
    = \  & \frac{k}{t+1} + \frac{1}{2} \Big( \frac{k}{t+1}\Big) \frac{1}{(t+1)/k - 1} \le  \frac{1.5k}{t+1},  
\end{align*}
where in the last step we used $(t+1)/k \ge 2$. Then, combing the above two relations yields that
\begin{align*}
    \sum_{k = 1}^{\lfloor t/2 \rfloor} \gamma^k \Big(\frac{t+1}{t + 1 - k}\Big)^\nu 
    & \le  \sum_{k = 1}^{\lfloor t/2 \rfloor} \gamma^k \exp \Big(\frac{1.5 k \nu}{t + 1}\Big)  = \sum_{k = 1}^{\lfloor t/2 \rfloor} \exp \Big(\frac{1.5 k \nu}{t + 1} + k\ln \gamma\Big). 
\end{align*}
Using $t > \frac{3\nu}{\ln (1/\gamma)} - 1$, $\frac{1.5 k \nu}{t + 1} + k\ln \gamma < \frac{k}{2}\ln \gamma$. 
It follows that
\begin{align*}
    \sum_{k = 1}^{\lfloor t/2 \rfloor} \gamma^k \Big(\frac{t+1}{t + 1 - k}\Big)^\nu < \sum_{k=1}^{\lfloor t/2 \rfloor} \gamma^{\frac{k}{2}} < \frac{\sqrt{\gamma}}{1 - \sqrt{\gamma}}.   
\end{align*}
Then, we consider the second half of the summation in \eqref{eq:bounded_sum}, i.e.,  $k \ge \lfloor t/2 \rfloor + 1$, 
\begin{align*}
    \sum_{k = \lfloor t/2 \rfloor + 1}^{t} \gamma^k \Big(\frac{t+1}{t + 1 - k}\Big)^\nu & \le \frac{t+1}{2} \gamma^{\lfloor t/2 \rfloor + 1} (t + 1)^\nu \\
    & \le \frac{1}{2} (t + 1)^{1 + \nu} \gamma^{t/2}   \le \sup_{t \ge 1} \frac{1}{2} (t+1)^{1 + \nu} \gamma^{t/2}.
\end{align*}
By setting $ c_{\gamma, \nu} := \max \Big(c_0, \ \frac{\sqrt{\gamma}}{1 - \sqrt{\gamma}} + \sup_{t \ge 1} \frac{1}{2} (t+1)^{1 + \nu} \gamma^{t/2} \Big)$, 
we obtain the desired upper bound in \eqref{eq:bounded_sum} and establish the lemma.
\end{proof} 

\begin{lemma}[Theorem 1 in \cite{nicholson1979eigenvalue}]
\label{lm:ab+ba}
Let $\bA$ and $\bB$ be positive definite $m$ by $m$ Hermitian matrices with eigenvalues $a_1 \ge a_2 \ge \ldots \ge a_m $, and $b_1 \ge b_2 \ge \ldots \ge b_m$, respectively, and let $\bQ = \bA\bB + \bB\bA$. Then the minimum eigenvalue $ q_m$ of $\bQ$ obeys the inequality, 
\begin{align}
\label{eq:cnbd}
    q_m \ge \min _{r=1, \ldots, m }\left[\frac{a_r b_r}{2}  \frac{\left(\sqrt{k_a}+1\right)^2 - k_b \left(\sqrt{k_a}-1\right)^2}{\sqrt{k_a}}  \right],
\end{align}
where $k_a, k_b$ are the spectral condition numbers of $\bA$ and $\bB$, respectively. 
\end{lemma}

\begin{lemma}[Lemma A.1 in \cite{jakovetic2023nonlinear}]
\label{lm:finite_bound}
Consider the (deterministic) sequence 
\begin{align*}
    v^{t+1}=\left(1-\frac{a_1}{(t+1)^\delta}\right) v^t+\frac{a_2}{(t+1)^\delta}, t \geq \tilde{t} ,
\end{align*}
with $a_1, a_2 > 0$, $0 < \delta \le 1, \tilde{t} > 0$, and $v^{\tilde{t}} \ge 0$. Further, assume $\tilde{t}$ is such that $a_1 /(t+1)^{\delta} \le 1$ for all $t \ge \tilde{t}$. Then $\lim_{t \rightarrow \infty} v^t = a_2/a_1$. 
\end{lemma}


\begin{lemma}[Lemma 3, \cite{polyak1987introduction} page 45]
\label{lm:cvg_diff_main}
Let $v^t$ be a nonnegative (deterministic) sequence satisfying $v^{t+1} \leq\left(1-r_1^t\right) v^t+r_2^t$,
for all $t \ge t'$, for some $t' \ge 0$. Here $\{r_1^t\}$ and $\{r_2^t\}$ are deterministic sequences satisfying
\begin{align*}
    0 < r_1^t \le 1, \quad \sum_{t = t'}^\infty r_1^t = \infty, \quad r_2^t \ge 0, \quad \lim_{t \rightarrow \infty} \frac{r_2^t}{r_1^t} =  0. 
\end{align*}
Then, $\lim_{t \rightarrow \infty} v^t = 0$. 
\end{lemma}


\bibliographystyle{IEEEtran}
\bibliography{refs}
        
\end{document}